\documentclass[10pt, leqno]{article}
\usepackage{latexsym,amsmath,amssymb}
\usepackage{graphicx}
\usepackage{color}

\newcommand{\fl}{\longrightarrow}
\newfont{\bb}{msbm10 at 12pt}

\def\s{\hbox{\bb S}}

\def\r{\mathbb{R}}

\def\n{{\cal N}}
\def\u{{\cal U}}

\newcommand{\ee}{\begin{equation}}

\newcommand{\fe}{\end{equation}}

\usepackage[latin1]{inputenc}
\usepackage {amsmath}
\usepackage {amsthm}
\usepackage{times}
\usepackage{amscd}
\usepackage{epsf}

\begin{document}

\theoremstyle{plain}\newtheorem{lem}{Lemma}
\theoremstyle{plain}\newtheorem{pro}{Proposition}
\theoremstyle{plain}\newtheorem{teo}{Theorem}
\theoremstyle{plain}\newtheorem{eje}{Example}[section]
\theoremstyle{plain}\newtheorem{no}{Remark}
\theoremstyle{plain}\newtheorem{cor}{Corollary}
\theoremstyle{plain}\newtheorem{defi}{Definition}

\begin{center}
\rule{15cm}{1.5pt} \vspace{.6cm}

{\Large \bf Entire solutions of the degenerateMonge-Ampère\\[3mm] equation with a finite number of singularities.}\\[5mm]

\rule{15cm}{1.5pt}\vspace{.5cm}\\

José A. Gálvez$^{a}$, Barbara Nelli$^b\ $\footnote{The authors were partially supported by INdAM-GNSAGA, PRIN-2010NNBZ78-009, MICINN-FEDER Grant No. MTM2013-43970-P, and Junta de Andalucía
Grant No. FQM325.}\vspace{0.8cm}
\end{center}

\noindent $\mbox{}^a$ Departamento de Geometr\'{\i}a y Topolog\'{\i}a, Facultad de Ciencias, Universidad de Granada,
E-18071 Granada, Spain; e-mail: jagalvez@ugr.es \vspace{0.2cm}

\noindent $\mbox{}^b$ Dipartimento Ingegneria e Scienze dell'Informazione e Matematica, Universitá di L'Aquila, via Vetoio - Loc. Coppito, E-67010 L'Aquila, Italy;
e-mail: nelli@univaq.it \vspace{0.3cm}

\begin{abstract}
We determine the global behavior of every ${\cal C}^2$-solution to the two-dimensional degenerate Monge-Ampère equation, $u_{xx}u_{yy}-u_{xy}^2=0$, over the finitely punctured plane. With this, we classify every solution in the once or twice punctured plane. Moreover, when we have more than two singularities, if the solution $u$ is not linear in a half-strip, we obtain that the singularities are placed  at  the vertices of a convex polyhedron $P$ and the graph of $u$ is made by pieces of cones outside of $P$ which are suitably glued along the sides of the polyhedron.
Finally, if we  look for analytic solutions, then there is at most one singularity and the graph of $u$ is either a cylinder (no singularity) or a cone (one singularity).
\end{abstract}


\section{Introduction.}

A celebrated result proved by A. V. Pogorelov \cite{P}, and independently by P. Hartman and L. Nirenberg \cite{HN}, states that all the global solutions to the degenerate Monge-Ampère equation
\begin{equation}
\label{the-equation}
u_{xx}u_{yy}-u_{xy}^2=0,
\end{equation}
where $u:\r^2\fl\r$ is a function of class ${\cal C}^2$,  are given by
$$
u(x,y)=\alpha(x)+c_0\,y,
$$
up to a rotation in the $(x,y)$-plane,  that is,  the graph of $u$   is a cylinder (see also \cite{S,Sa1}).

This degenerate Monge-Ampère equation has been extensively studied from an analytic point of view and also from a geometric point of view since the graph of every solution determines a flat surface in the Euclidean 3-space.

Local properties of the solutions of (\ref{the-equation}) have been analyzed in many papers (see, for instance, \cite{Sa2,U} and references therein). Our objective is to study the global behavior of the solutions of (\ref{the-equation}) for the largest non simply-connected domains, that is, in the finitely punctured plane.

Observe that, when the Monge-Ampère equation is elliptic or hyperbolic, a large amount of work in the understanding of these solutions in the punctured plane has been achieved from a local and global point of view (see, among others, \cite{ACG, B1,B2, GHM,GJM,GMM, GM, HB, JX, J,Mi1,Mi2,SW}).

The paper is organized as follows. After a first section of preliminaries, in Section \ref{s3} we establish the global behavior of every ${\cal
C}^2$ solution $u:\r^2\backslash\{p_1,\ldots,p_n\}\fl\r$ to the Monge-Ampère equation (\ref{the-equation}) with isolated singularities at the points $p_i$. We show that for every singular point $p_i$ there exists at least a sector $S_i\subseteq\r^2\backslash\{p_1,\ldots,p_n\}$ such that the graph of $u$ over $S_i$ is a piece of a cone. Moreover, there exists at most a maximal strip $S$ such that the graph of $u$ over $S$ is a cylinder. As a consequence of these results, we obtain that if $u$ is an analytic solution then there is at most one singularity, and the graph of $u$ is either a cylinder if there is no singularity or a cone if there is one singularity.

In Section \ref{s4} we classify all solutions to (\ref{the-equation}) with one or two singularities. In particular, if  $u$ is a solution with one singularity at $p_0$ then, either the graph of $u$ is a cone or there exists a line $r$ containing to $p_0$ such that the graph of $u$ is a cylinder over one half-plane determined by $r$ and a cone over the other half-plane. We also describe every solution in the twice punctured plane.

Although the behavior of a solution $u$ to (\ref{the-equation}) with more than two singularities can be complicated, due to the existence of some regions in the $(x,y)-$plane where $u$ is a linear function, we classify in Section \ref{s5} every solution which does not admit a large region where $u$ is linear, that is, every solution such that the domain where $u$ is linear does not contain a half-strip of $\r^2$. In such  case, we show that the singular points are the vertices of a convex compact polyhedron $C\subseteq\r^2$ with non empty interior, the solution $u$ must be linear over $C$ and the graph of $u$ is made by some cones over $\r^2\backslash C$ which are suitably glued along the segments of the boundary of the planar polyhedron $u(C)$.

\section{Preliminaries.}

Let $\Omega\subseteq\r^2$ be a domain, and $u:\Omega\fl\r$ be a function of class ${\cal C}^2$ satisfying the degenerate Monge-Ampère equation
$u_{xx}u_{yy}-u_{xy}^2=0,$ in $\Omega.$
As we mentioned in the introduction, it is well known that every solution $u(x,y)$ to the degenerate Monge-Ampère equation in the whole plane $\r^2$ is given by
$u(x,y)=\alpha(x)+c_0\,y,$ up to a rotation in the $(x,y)$-plane (see \cite{HN, Ma, P}), that is, its graph is a cylinder.

Thus, it is natural to study the solutions to the degenerate Monge-Amp\`ere equation in the possible largest domains. In other words, we consider solutions to (\ref{the-equation}) in the non simply-connected domains $\r^2\backslash\{p_1,\ldots,p_n\}$. Here, we assume $u$ has an {\it isolated singularity} at any $p_i$, that is, $u$ cannot be ${\cal C}^2$-extended to $p_i$.

For the study of these solutions, we will need the following result which is a consequence of \cite[Lemma 2]{HN}.

\begin{lem}\label{HN}
Let $u(x,y)$ be a ${\cal C}^2$ solution to the degenerate Monge-Ampère equation (\ref{the-equation}) and $(x_0,y_0)\in\Omega$ a point where the Hessian matrix of $u$ does not vanish. Then, there exists a line $r$ in the $(x,y)-$plane such that the connected component $r_0$ of $r\cap\Omega$ containing $(x_0,y_0)$ satisfies:
\begin{enumerate}
\item $r_0$ is made of points with non vanishing Hessian matrix,
\item the gradient of $u$ is constant on $r_0$.
\end{enumerate}
Moreover, there exists an open set $U\subseteq\Omega$ containing $r_0$ such that if a point $(x_1,y_1)\in U$ has the same gradient as a point of $r_0$ then $(x_1,y_1)\in r_0$.
\end{lem}

From now on, given a solution $u$ to (\ref{the-equation}) in $\Omega$, we will denote by $\n$ the open set given by the points in $\Omega$ whose Hessian matrix does not vanish, and by $\u=\Omega\backslash\n$.

Note that the graph $\Sigma$ of a solution $u$ to (\ref{the-equation}) corresponds, from a geometric point of view, to a flat surface in $\r^3$, and the points
$$
\Sigma_{\n}:=\{(p,u(p)):\ p\in\n\},\qquad \Sigma_{\u}:=\{(p,u(p)):\ p\in\u\},
$$
correspond, respectively, to the non umbilical points and to the umbilical points of the graph $\Sigma$.

Given a point $p\in\n$ we will denote by $r(p)$ the piece of line $r_0$ determined by Lemma \ref{HN}. Moreover, since the gradient of $u$ is constant on $r(p)$, the image of $r(p)$ given by
$$
R(p):=\{(q,u(q))\in\r^3:\ q\in r(p)\}
$$
is a piece of line in $\r^3$.

Thus, the previous lemma asserts that $\Sigma_{\n}$ is made of pieces of pairwise disjoint lines of $\r^3$. In addition, every connected component of $\Sigma_{\u}$ with an interior point is contained in a plane, because its Hessian matrix vanishes identically.

A first consequence of the previous lemma is that each solution of the degenerate Monge-Ampère equation can be continuously extended to the singularity because the norm of its gradient is bounded around the singularity.

\begin{lem}\label{l2}
Let $\Omega\subseteq\r^2$ be a domain, $p_0\in\Omega$ and $u$ be a ${\cal C}^2$ solution of (\ref{the-equation}) in $\Omega\backslash\{p_0\}$. Then the modulus of the gradient of $u$ is bounded in a neighborhood of $p_0$. In particular, $u$ can be continuously extended to $p_0$.
\end{lem}
\begin{proof}
Consider $\varepsilon>0$ such that the closed disk $\overline{D}_{p_0}(\varepsilon)$ centered at $p_0$ with radius $\varepsilon$ is contained in $\Omega$. Let
$$
m=\max\{\|grad_p(u)\|:\ \|p-p_0\|=\varepsilon\},
$$
where $grad_p(u)$ denotes the gradient of $u$ at a point $p$.

Let us see that the modulus of the gradient of $u$ is less than or equal to $m$ for every point $p\in\overline{D}_{p_0}(\varepsilon)\backslash\{p_0\}$.

Let $p\in\n$, then the piece of line $r(p)$ has at least a point $q$ in the boundary of the disk $\overline{D}_{p_0}(\varepsilon)$. So, from Lemma \ref{HN} one has that $\|grad_p(u)\|=\|grad_q(u)\|\leq m$. Hence, the previous inequality happens in the closure of $\n$.

On the other hand, if $p$ is a point in the interior of $\u$, we can  consider the connected component $\u_0$ of $\u$ to which $p$ belongs. The Hessian matrix of $u$ vanishes identically in $\u_0$, so, the gradient of $u$ is constant in the closure of $\u_0$. Moreover, since the closure of $\u_0$ intersects the closure of $\n$, we obtain that $\|grad_p(u)\|\leq m$, as we wanted to show.
\end{proof}

\section{Global behaviour of the graphs.}\label{s3}

Now, consider a ${\cal C}^2$ solution $u:\r^2\backslash\{p_1,\ldots,p_n\}\fl\r$ to the degenerate Monge-Ampère equation (\ref{the-equation}) with singularities at the points $p_i$. Let $p\in\n$, then from the previous considerations, we know that $r(p)$ must be:
\begin{enumerate}
\item a line, or
\item a half-line with end point at a singular point $p_i$, or
\item a segment with end points at two different singular points $p_i,p_j$.
\end{enumerate}

As a consequence we obtain the following result.

\begin{pro}\label{max-strip}
Let $u:\r^2\backslash\{p_1,\ldots,p_n\}\fl\r$ be a ${\cal C}^2$ solution of (\ref{the-equation}) with isolated singularities at the points $p_i$. Assume there exists a point $p\in\n$ such that $r(p)$ is a line and let $S(p)\subseteq\r^2\backslash\{p_1,\ldots,p_n\}$ be the maximal open strip containing $r(p)$. Then $u$ is given in $S(p)$ by
\ee\label{cilindro}
u(x e_1+y e_2)=\alpha(x)+v_0\,y,\qquad y\in\r
\fe
for a ${\cal C}^2$ function $\alpha(x)$, a constant $v_0\in\r$, and an orthonormal basis $\{e_1,e_2\}$ of $\r^2$ where $e_2$ is parallel to $r(p)$.

Moreover,
\begin{enumerate}
\item[1)] if $q\in\n$ satisfies that $r(q)$ is a line then $r(q)\subseteq S(p)$,
\item[2)] if $r$ is a line contained in $\u$ then $r\subseteq S(p)$.
\end{enumerate}
\end{pro}
\begin{proof}
Let $q$ be a point in the open strip $S(p)$, with $q\in\n$. Since $S(p)$ has no singular point then $r(q)$ intersects $r(p)$ or $r(q)$ is parallel to $r(p)$. From Lemma \ref{HN}, in the former case $r(q)=r(p)$ and in the latter case $r(q)$ must be a line contained in $S(p)$.

Now, let $r_0$ be a parallel line to $r(p)$ contained in $S(p)$. Then, from the previous discussion we have obtained that  $r_0\subseteq\n$ or $r_0\subseteq\u$. As a consequence, if $r_0$ has a point $q$ in the interior of $\u$, $int(\u)$, then $r_0$ is contained in $int(\u)$.

Hence, if $r_0=r(q)\subseteq\n$ then the image of $r_0$ is the line $R(q)$, and if $r_0\subseteq int(\u)$ then the image of $r_0$ is a line, because the image of each connected component of $int(\u)$ is a piece of a plane. Thus, the image of  $r_0$ must also be a line if  $r_0\subseteq \partial\n=\partial (int(\u))$.

Therefore, $u$ is given in $S(p)$ by
$$
u(x e_1+y e_2)=\alpha(x)+v(x)\,y,\qquad y\in\r
$$
for certain functions $\alpha(x), v(x)$, and $\{e_1,e_2\}$ orthonormal basis of $\r^2$, with $e_2$ parallel to $r(p)$. In addition, from (\ref{the-equation}), $v(x)$ must be a constant $v_0\in\r$.


Let us prove now the case 1), that is, assume $q\in\n$ such that $r(q)$ is a line and see that $r(q)\subseteq S(p)$. If $r(q)\not\subseteq S(p)$, then $r(q)$ is parallel to $r(p)$ because otherwise $r(q)$ intersects $r(p)$, and so $r(p)=r(q)$.

Denote by $S_0$ the unique connected component of $\r^2\backslash S(p)$ containing $r(q)$. Observe   that $S_0$ is a closed half-plane of $\r^2$, parallel to $r(p)$, so that its boundary $\partial S_0$ is a line containing at least a singular point $p_{i_0}$, due to the maximality of $S(p)$.

Let $\varepsilon>0$ such that the open disk $D_{p_{i_0}}(\varepsilon)$ centered at $p_{i_0}$ with radius $\varepsilon$ does not contain another singular point, and $D_{p_{i_0}}(\varepsilon)\cap r(q)=\emptyset$. Moreover, since the number of singular points is finite,  we can choose $\varepsilon$ small enough so that if $r$ is any line parallel to $r(p)$ intersecting $D_{p_{i_0}}(\varepsilon)$ then $r=\partial S_0$ or $r$ has no singular point.

The set $D_{p_{i_0}}^+(\varepsilon)=D_{p_{i_0}}(\varepsilon)\cap int(S_0)$ must contain some point $a\in\n$. Otherwise, the image of $D_{p_{i_0}}^+$ would be a piece of a plane, and from (\ref{cilindro}) one has that $p_{i_0}$ is a singular point if, and only if, every point in $D_{p_{i_0}}\cap\partial S_0$ is singular. This is a contradiction which claims the existence of the previous point $a$.

Since $\n$ is an open set and the set of singularities is finite, we can assume that $r(a)$ is not a segment joining two singular points. So, $r(a)$ is a half-line or a line which does not intersects $r(p)$ or $r(q)$, that is, $r(a)$ must be parallel to $r(p)$. Then, from the choice of $\varepsilon$, one has that $r(a)$ is also a line parallel to $r(p)$.

Finally, if we choose the maximal strip $S(a)\subseteq\r^2\backslash\{p_1,\ldots,p_n\}$ containing the line $r(a)$, we obtain from the choice of $\varepsilon$ that $\partial S(p)\cap\partial S(a)$ is the line parallel to $r(p)$ containing $p_{i_0}$. But this is not possible, because in this case one has from (\ref{cilindro}) that  $p_{i_0}$ is  a singular point if and only if every point in $\partial S(p)\cap\partial S(a)$ is singular.

Therefore, there is no $q\in\n$ such that $r(q)$ is a line which is not contained in $S(p)$.

In the case 2), the line $r\subseteq\u$ must be parallel to $r(p)$ because otherwise $r\cap r(p)\neq\emptyset$, contradicting Lemma \ref{HN}. So, this case can be exactly proven as the previous case replacing $r(q)$ by $r$.
\end{proof}

The previous result asserts that the graph of $u$ over $S(p)$ is a cylinder in $\r^3$ foliated by lines, all of them parallel to $R(p)$. Observe that this cylinder has non umbilical points, and probably umbilical points as well. In particular, if there is no singular point then $S(p)=\r^2$ and $u$ is globally determined by (\ref{cilindro}).

\begin{defi}{\em 
We say that $u$ is a {\em cylindrical function} in a strip $S$ if $u$ is given by (\ref{cilindro}) in $S$. Moreover, we say that $u$ is a {\em conical function} in an open sector
$$
S_{\theta_1}^{\theta_2}(p)=\{p+\rho\, (\cos\theta,\sin\theta):\ \rho>0,\ \theta_1<\theta<\theta_2\},\qquad 0<\theta_2-\theta_1<2\pi,
$$
or in $\r^2\backslash\{p\}$, if $u$ is given by
\begin{equation}\label{cono}
u(p+\rho\, (\cos\theta,\sin\theta))=u_0+\rho \,\alpha(\theta)
\end{equation}
in one of the previous domains, for a certain ${\cal C}^2$ function $\alpha(\theta)$.}
\end{defi}

Now, let us prove some technical lemmas which will be necessary in our study.

\begin{lem}\label{l3}
Let $\Omega$ be a domain, $p_0\in\Omega$, $S\subseteq\Omega$ be an open sector with vertex at $p_0$, and $u:\Omega\backslash\{p_0\}\fl\r$ be a ${\cal C}^2$ solution of (\ref{the-equation}) with isolated singularity at $p_0$. If for every $p\in S\cap\n$ one has that $r(p)$ is a half-line with end point at $p_0$ then $u$ is a conical function in $S$.
\end{lem}
\begin{proof}
If $p\in S\cap\n$ then the graph of $r(p)$ is the half-line $R(p)$, and if $r_0\subseteq int(\u)\cap S$ is a half-line with end point at $p_0$ then the image of $r_0$ is a half-line in $\r^3$, because each connected component of the graph of $int(\u)$ is a piece of a plane. As a consequence, the image of a half-line  $r_0$ with end point at $p_0$ is also a half-line of $\r^3$ if $r_0\subseteq \partial\n=\partial (int(\u))$.

Consequently, $u(p_0+\rho\, \theta)=u_0(\theta)+\rho \,\alpha(\theta)$, with $\rho>0$. And, from Lemma \ref{l2}, $u_0(\theta)$ must be constant since $u$ is continuous at $p_0$.
\end{proof}
\begin{lem}\label{l4}
Let $u:\r^2\backslash\{p_1,\ldots,p_n\}\fl\r$ be a ${\cal C}^2$ solution of (\ref{the-equation}) with isolated singularities at the points $p_i$. Let $S_1,S_2\subseteq\r^2\backslash\{p_1,\ldots,p_n\}$ be two disjoint open sets such that
\begin{enumerate}
\item[1)] $u$ is a cylindrical function in the strip $S_1$ and $u$ is a conical function in the sector $S_2$;  or
\item[2)] $u$ is a conical function in the sector $S_1$ and also in the sector $S_2$, both with different vertices.
\end{enumerate}

If $\partial S_1\cap\partial S_2$ contains a segment, then every non singular point in $\partial S_1\cap\partial S_2$ is contained in $\u$.
\end{lem}
\begin{proof}
If $S$ is a strip and $e$ is a normal unit vector to a connected component of $\partial S$ then, from (\ref{cilindro}), one has that $Hess(u)(e,e)$ is constant for all point in this component.

On the other hand, if $S$ is a sector with vertex at $p_0$ and the half-line $p_0+\rho\,\theta_0$, with $\rho>0$, is contained in $\partial S$ then, from (\ref{cono}), one has that $Hess(u)(e,e)=\frac{\alpha''(\theta_0)+\alpha(\theta_0)}{\rho}$, where $e$ is a normal unit vector to the half-line.

In the case 1) or 2), it is clear that $\alpha''(\theta_0)+\alpha(\theta_0)$ vanishes at every point in $\partial S_1\cap\partial S_2$. This is equivalent to show that the Hessian matrix of $u$ vanishes at every non singular point in $\partial S_1\cap\partial S_2$.
\end{proof}

\begin{pro}\label{l5}
Let $u:\r^2\backslash\{p_1,\ldots,p_n\}\fl\r$ be a ${\cal C}^2$ solution of (\ref{the-equation}) with isolated singularities at the points $p_i$, and $p\in\n$. Then $r(p)$ cannot be a segment.

Moreover, for any singular point $p_i$,
\begin{equation}\label{npi}
\n_{p_i}=\{p\in\n:\ r(p) \text{ is a half-line with end point }p_i\}
\end{equation}
is an open set.
\end{pro}
\begin{proof}
Let $p\in\n$ such that $r(p)$ is either a segment or a half-line. Up to a translation and a rotation, we can assume $p$ is the origin and a singular end point of $r(p)$ is $(x_1,0)$ with $x_1<0$.

Since $\n$ is an open set, we consider $\varepsilon>0$ such that the open disk $D(\varepsilon)$ centered at the origin with radius $\varepsilon$ is contained in $\n$. Moreover, using that the set of singular points is finite, $\varepsilon$ can be chosen small enough in such a way there is no singular point $p_k=(x_k,y_k)$ with $0<|y_k|<\varepsilon$.

Let $q_n=(x_n,y_n)$ be a sequence in $D(\varepsilon)^+=\{(x,y)\in D(\varepsilon):y>0\}$ going to $p=(0,0)$. Observe that the number of segments of the form $r(q)$ contained in $\n$ can only be finite because the number of singularities is finite. Thus, replacing $q_n$ by a point close to it, we can assume $r(q_n)$ is a half-line or a line, for every point $q_n$ of the sequence. Since $r(q_n)$ cannot intersect $r(p)$, the angle between $r(q_n)$ and $r(p)$ tends to zero.

If there exists a subsequence $q_m$ such that $r(q_m)$ are lines, then all of them must be parallel and, from Proposition \ref{max-strip}, $r(q)$ is a horizontal line for all $q=(x,y)\in D(\varepsilon)^+$.

Otherwise, since the number of singularities is finite, there exists a subsequence $q_m$ such that $r(q_m)$ are half-lines with the same singular end point $p_{i_1}$. Thus, as the angle between $r(q_m)$ and $r(p)$ tends to zero, $p_{i_1}$ is necessarily placed on the line $y=0$. In particular, taking $\varepsilon$ small enough, we can assume that the sector made of all half-lines containing a point $q\in D(\varepsilon)^+$ and with end point at $p_{i_1}$ does not contain another singular point. Hence, each half-line of this sector with end point $p_{i_1}$ is of the form $r(q)$ for $q\in D(\varepsilon)^+$.

Analogously, we can follow the same reasoning for $D(\varepsilon)^-=\{(x,y)\in D(\varepsilon):y<0\}$. Thus, $r(q)$ is a  horizontal line for all  $q\in D(\varepsilon)^-$, or $r(q)$ is a half-line for all $q\in D(\varepsilon)^-$ with the same singular end point $p_{i_2}$, placed at the line $y=0$.

Now, let us denote by $S^+$ the strip $\{(x,y)\in\r^2:0<y<\varepsilon\}$ if $r(q)$ is a horizontal line for all $q\in D(\varepsilon)^+$, or the sector containing $D(\varepsilon)^+$ with vertex at $p_{i_1}$ if $r(q)$ is a half-line for all $q\in D(\varepsilon)^+$. We also denote $S^-$ the corresponding set for the points in $D(\varepsilon)^-$.

From Proposition \ref{max-strip}, $S^+$ and $S^-$ cannot be two strips because the point $(x_1,0)$ is singular. Moreover, from Lemma \ref{l4}, $S^+$ and $S^-$ are two sectors with the same vertex $p_{i_1}=p_{i_2}$ because otherwise $p=(0,0)$ would be contained in $\u$. Thus, from Lemma \ref{l3}, if $r$ is the half-line containing $p=(0,0)$ with end point $p_{i_1}=p_{i_2}$ then a point in $r$ is a singular point if and only if every point in $r$ is singular. Hence, $p_{i_1}=p_{i_2}=(x_1,0)$ and $r(p)$ cannot be a segment.

This proves that if $p\in\n$ such that $r(p)$ is a half-line then there exists $\varepsilon>0$ such that $D_p(\varepsilon)\subseteq\n$ and for all $q\in D_p(\varepsilon)$ the set $r(q)$ is a half-line with the same end point. That is, $\n_{p_i}$ is an open set.
\end{proof}

As a consequence of Propositions \ref{max-strip} and \ref{l5},  we get the following result.

\begin{cor}
\label{c0}
Let $u:\r^2\backslash\{p_1,\ldots,p_n\}\fl\r$ be a ${\cal C}^2$ solution of (\ref{the-equation}) with isolated singularities at the points $p_i$. If $V$ is a connected component of
$\n$ then either  $V$ is made of parallel lines or $V$ is made of half-lines with  common endpoint at some $p_i.$
\end{cor}

\begin{cor}\label{c1}
Let $u:\r^2\backslash\{p_1,\ldots,p_n\}\fl\r$ be a ${\cal C}^2$ solution of (\ref{the-equation}) with isolated singularities at the points $p_i$. Then the set $\n_{p_i}$, given by (\ref{npi}), is a non empty union of open sectors with vertex at $p_i$, for every singularity $p_i$. In particular, $u$ is a conical function in each connected component of $\n_{p_i}$.
%
\end{cor}
\begin{proof}
From Proposition \ref{l5} and Lemma \ref{l3}, we only need to show that $\n_{p_i}$ is non empty for each singular point $p_i$. So, fix a singular point $p_i$ and assume $\n_{p_i}$ is empty.

If there existed a punctured  neighborhood $U$ of $p_i$ contained in $\u$ then the image of $U$ would lie in a plane and $p_i$ would not be a singular point. Thus, there must exist a sequence $q_n\in\n$ going to $p_i$. Moreover, a subsequence $q_m$ must satisfy that $r(q_m)$ is a line for all $m$ (and so the lines are parallel) or $r(q_m)$ is a half-line for all $m$ with a common singular end point $p_{j}\neq p_{i}$. Up to a translation and a rotation, we can assume $p_i$ is the origin and the limit line of the lines containing $r(q_m)$ is $y=0$. In particular, in the latter case $p_j$ is placed in the line $y=0$.

Observe that $q_m$ cannot be placed in the segment joining the singular points $p_i$ and $p_j$, that is, in any case, every point $q_m$ is in the open upper half-space $\r^2_+=\{(x,y)\in\r^2:y>0\}$ or in the open lower half-space $\r^2_-=\{(x,y)\in\r^2:y<0\}$.

Passing to a subsequence, if necessary, we can assume $q_m$ is contained in $\r^2_+$ for all $m$, or  $q_m$ is contained in $\r^2_-$ for all $m$. Suppose $q_m\in\r^2_+$ for every $m$. If $r(q_m)$ is a line for each $m$ then we denote by $S^+$ the maximal strip determined by Proposition \ref{max-strip}, which satisfies that the line $y=0$ is in its closure. If $r(q_m)$ is a half-line for all $m$ then for $m$ big enough the open sector determined by $r(q_m)$ and $y=0$ does not contain any singular point, and we call $S^+$ to this sector. In the first case the function $u$ would be cylindrical in $S^+$ (Proposition \ref{max-strip}), and in the second case $u$ would be conical in $S^+$ (Lemma \ref{l3}) with vertex at $p_j\neq p_i$.

On the other hand, there must exist a sequence of points in $\n\cap\r^2_-$ going to $p_i$. Otherwise there would exist $\varepsilon>0$ such that the open half-disk $D(\varepsilon)\cap\r^2_-$ is contained in $\u$. Thus, the image of $D(\varepsilon)\cap\r^2_-$ lies in a plane. But, since $u$ is cylindrical or conical in $S^+$ then the origin $p_i$ is a singular point if and only if every point of a neighborhood of $p_i$ in the line $y=0$ is singular, which gives us a contradiction.
Thus, the previous sequence exists and we can define $S^-$ in an analogous way.

So, $u$ will be cylindrical or conical in $S^+$ and in $S^-$. But, in any case, this is again a contradiction because $p_i$ would be a singular point if and only if every point of a neighborhood of $p_i$ in the line $y=0$ is singular.  Hence $\n_{p_i}\not=\emptyset.$
\end{proof}

As a consequence we obtain a first global classification result.
\begin{teo}
Let $u:\r^2\backslash\{p_1,\ldots,p_n\}\fl\r$ be an analytic solution of (\ref{the-equation}) with isolated singularities at the points $p_i$.
Then, either there is no singular point and $u$ is a cylindrical function in $\r^2$, or there is a unique singular point and $u$ is a conical function in $\r^2\backslash\{p_1\}$.
\end{teo}
\begin{proof}
The result is well known if there is no singular point \cite{HN,P}. Otherwise, up to a translation, we can assume that the origin is a singular point. So, for polar coordinates $(\rho,\theta)$, the function $u(\rho,\theta)$ is analytic and, from the previous Corollary, $u_{\rho\rho}$ vanishes identically in the non empty open set $\n_{(0,0)}$. Therefore, $u_{\rho\rho}$ must vanish globally, that is, $u$ is a conical function in $\r^2\backslash\{(0,0)\}$.
\end{proof}

\section{Entire solutions with one or two singularities.}\label{s4}

We devote this Section to the characterization of the entire solutions to the degenerate Monge-Ampère equation $u_{xx}u_{yy}-u_{xy}^2=0$, where $u$ is a ${\cal C}^2$ function in $\r^2$ minus one or two points. These results will be a consequence of the global results studied in the previous Section.

We first prove that there only exist two different kinds of solutions in the punctured plane.
\begin{teo}\label{t2}
Let $u:\r^2\backslash\{(0,0)\}\fl\r$ be a ${\cal C}^2$ solution of (\ref{the-equation}) with an isolated singularity at the origin. Then, $u$ is given by one of the following cases (Figure \ref{fig1}):
\begin{enumerate}
\item $u$ is a conical function in $\r^2\backslash\{(0,0)\}$, or
\item there exists a line $r$ containing the origin such that $u$ is a cylindrical function in one half-plane determined by $r$, and $u$ is a conical function in the other half-plane.
\end{enumerate}
\end{teo}
\begin{proof}
From Corollary \ref{c1}, $\n$ cannot be empty. Thus, if there is no point $p\in\n$ such that $r(p)$ is a line then, from Lemma \ref{l3}, $u$ is a conical solution in $\r^2\backslash\{(0,0)\}$.

On the other hand, if there is a point $p\in\n$ such that $r(p)$ is a line then, from Proposition \ref{max-strip}, there exists an open half-plane $S(p)$, containing $r(p)$, such that $u$ is a cylindrical function in $S(p)$. Moreover, the boundary of $S(p)$ in $\r^2$ is a line $r$ containing the singular point $(0,0)$. If $S^-$ is the other open half-plane determined by $r$, again Proposition \ref{max-strip} gives that $r(q)$ must be a half-line for every point $q\in\n\cap S^-$. Hence, $u$ is a conical function in $S^-$, by Lemma \ref{l3}.
\end{proof}
\begin{figure}[!h]
\centerline{\includegraphics[scale =0.6]{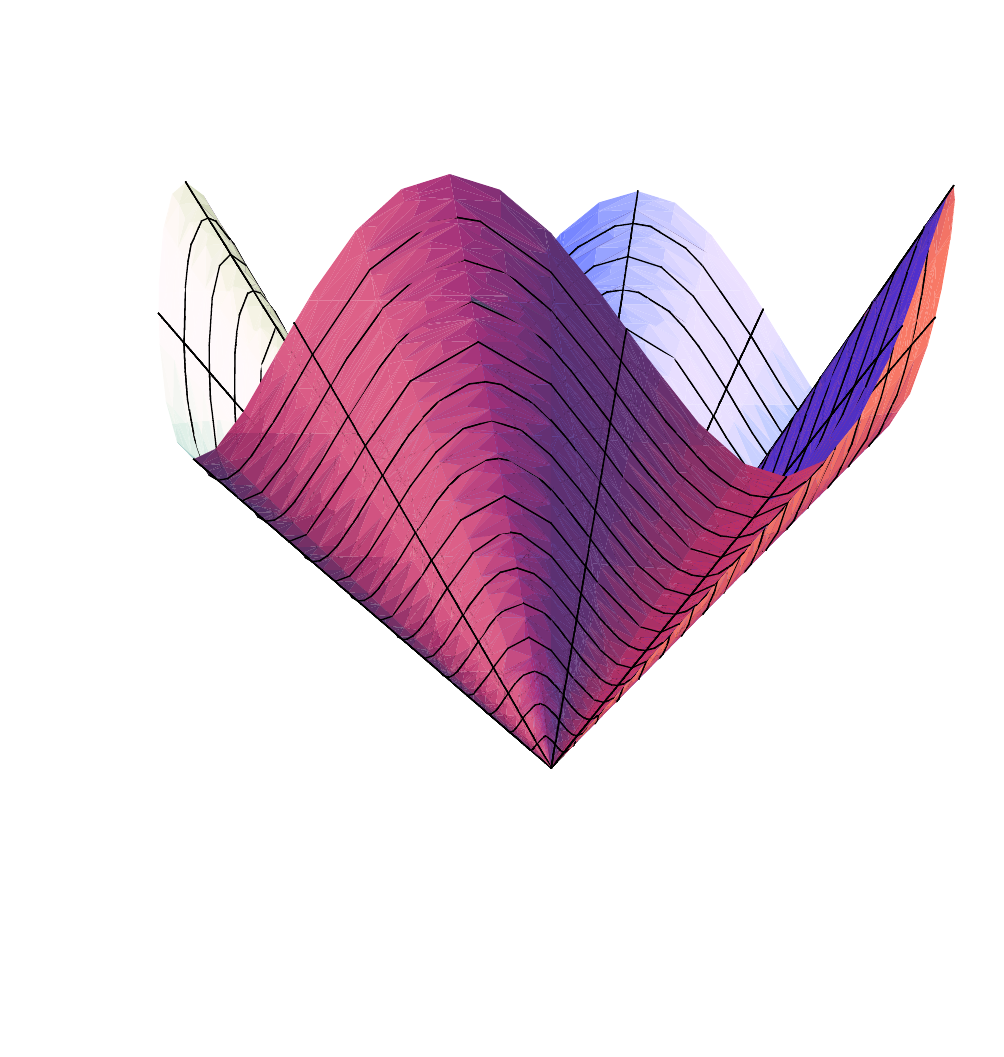} \hspace{2cm} \includegraphics[scale =0.6]{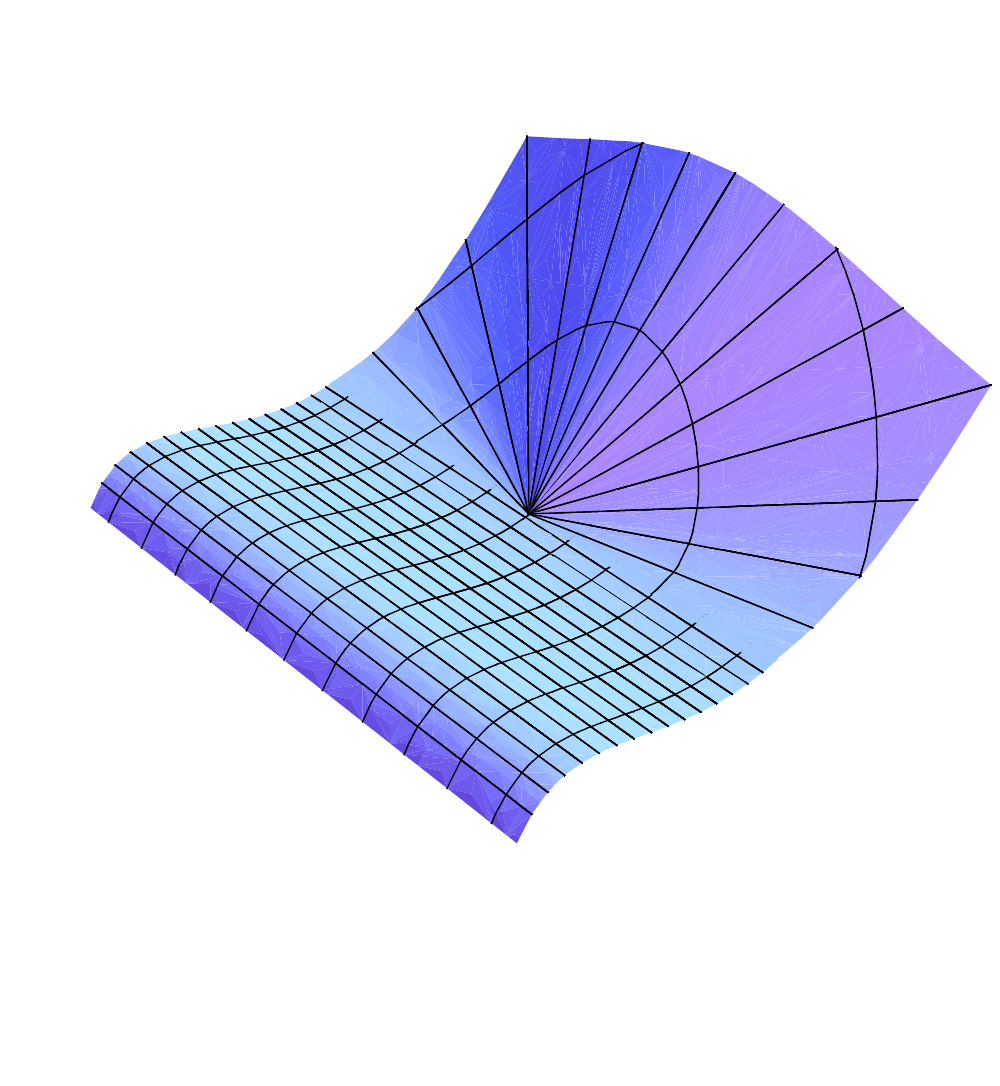}}\vspace{-1.5cm}
\caption[]{solutions with one singularity.}
\label{fig1}
\end{figure}

Now we present some families of examples of solutions to the equation (\ref{the-equation}) in the twice punctured plane.
\begin{eje}\label{ej1}{\em
Consider an open strip $S$ in $\r^2$ so that its boundary is given by two lines $r_1,r_2$. Choose two points $p_1\in r_1$, and $p_2\in r_2$. Then, a solution $u$ to the equation (\ref{the-equation}) can be obtained in $\r^2\backslash\{p_1,p_2\}$ such that $u$ is a cylindrical function in $S$, and $u$ is a conical function in each half-plane of $\r^2\backslash \overline{S}$ with vertex at $p_i$ (Figure \ref{segundo}).}
\end{eje}
\begin{figure}[!h]
\centerline{\includegraphics[scale =0.3]{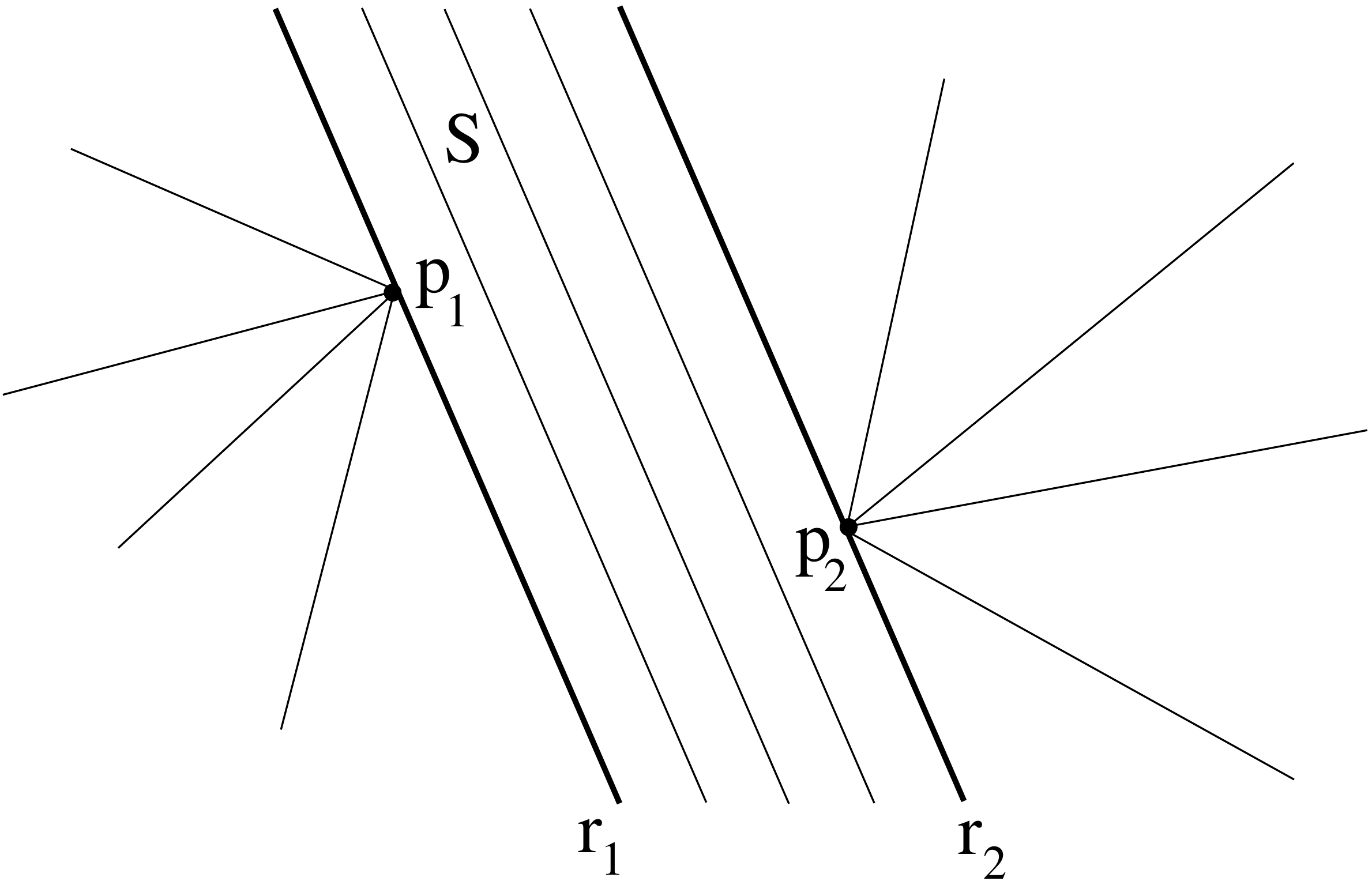} }
\caption[]{the configuration of $r(p)$ lines for example \ref{ej1}.}
\label{segundo}
\end{figure}

\begin{eje}\label{ej2}{\em 
Let $S$ be an open half-plane in $\r^2$ with boundary given by a line $r$, and $p_1,p_2\in r$. Let $h_1,h_2$ be two parallel half-lines contained in $\r^2\backslash \overline{S}$ with respective end points $p_1,p_2$. A solution $u$ to (\ref{the-equation}) can be defined in $\r^2\backslash\{p_1,p_2\}$ such that $u$ is a cylindrical function in $S$, $u$ is a linear function in the half-strip determined by $h_1,h_2$ and $r$, and $u$ is a conical function in the remaining two domains with vertices at $p_1,p_2$, respectively (Figure \ref{tercero}).}
\end{eje}
\begin{figure}[!h]
\centerline{\includegraphics[scale =0.3]{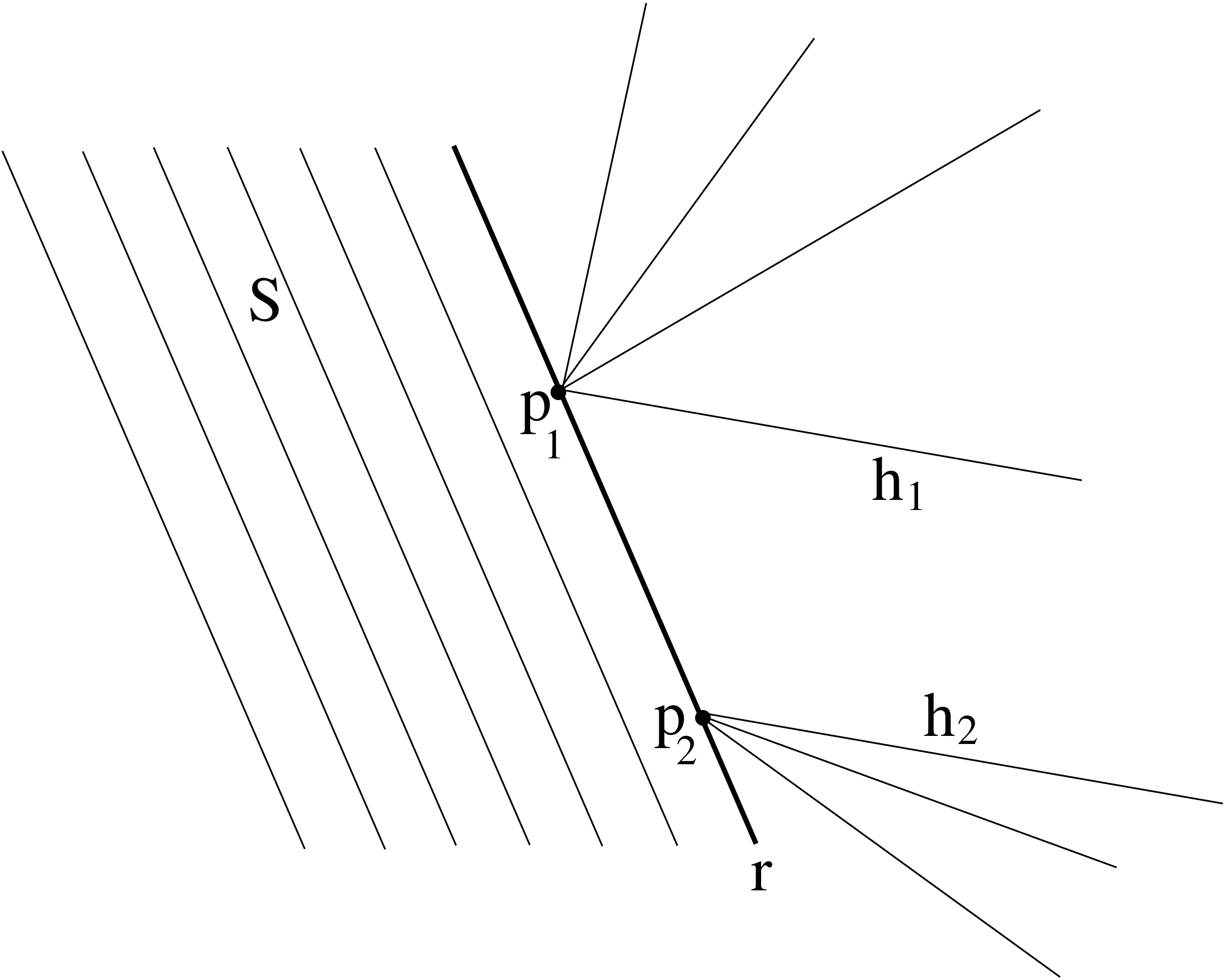} }
\caption[]{the configuration of $r(p)$ lines for example \ref{ej2}.}
\label{tercero}
\end{figure}

\begin{eje}\label{ej3}{\em 
Assume $S$ is an open half-plane in $\r^2$ with boundary given by a line $r$, $p_1\in r$, and $p_2\in\r^2\backslash \overline{S}$. Let $h_1,h_2$ be two half-lines in $\r^2\backslash S$ with end point at $p_1$ such that at most one of them is contained in $r$, and $p_2$ belongs to the closed sector determined by $h_1$ and $h_2$. Let $\overline{h_i}$ be the half-line parallel to $h_i$  with end point at $p_2$. Then, a solution $u$ to (\ref{the-equation}) can be defined in $\r^2\backslash\{p_1,p_2\}$ such that $u$ is a cylindrical function in $S$, $u$ is a linear function in the unbounded domain determined by $h_1,h_2,\overline{h_1}$ and $\overline{h_2}$, and $u$ is a conical function in the rest of domains  (Figure \ref{cuarto}).}
\end{eje}
\begin{figure}[!h]
\centerline{\includegraphics[scale =0.3]{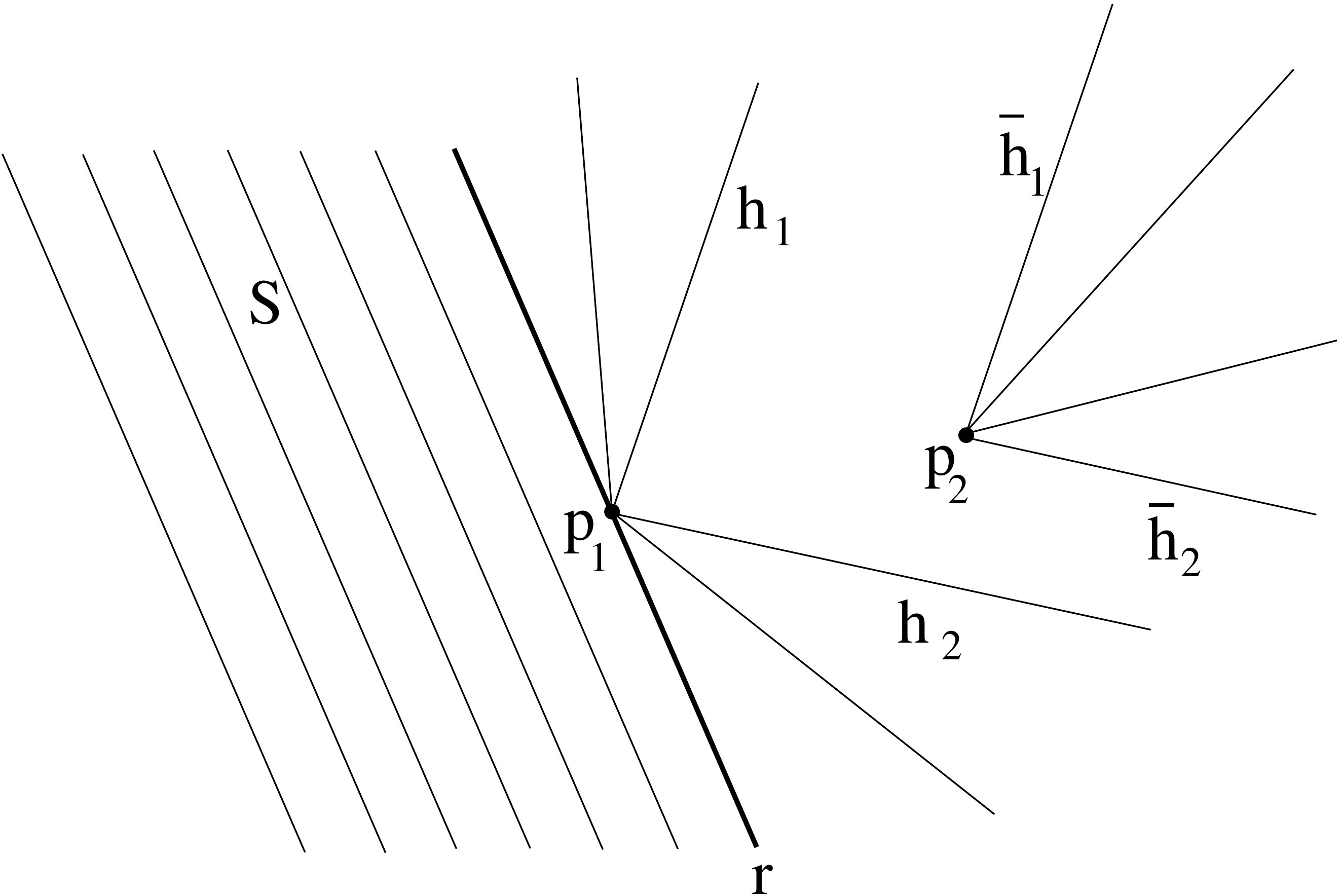}\hspace{1cm} \includegraphics[scale =0.3]{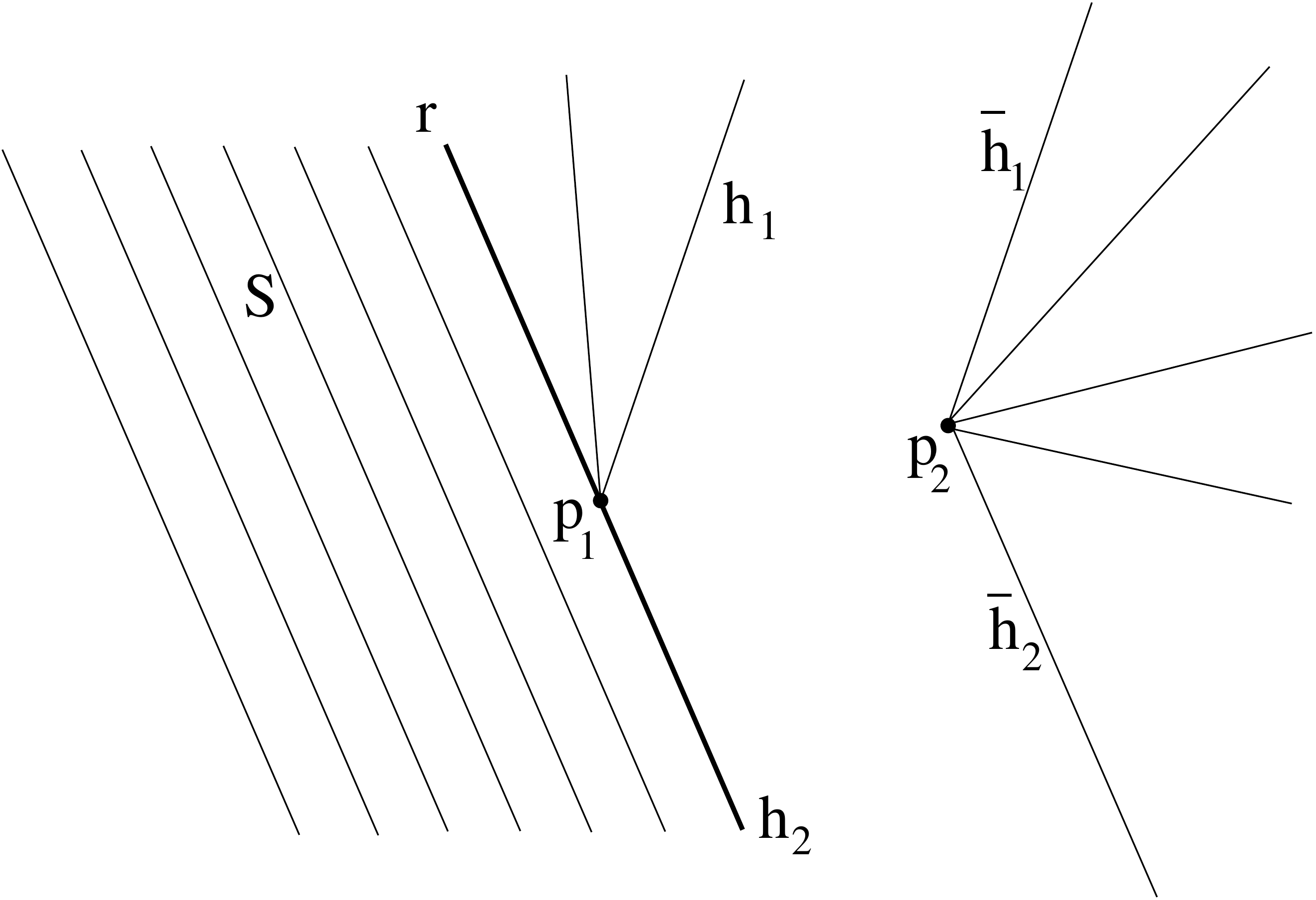} }
\caption[]{the configurations of $r(p)$ lines for example \ref{ej3}.}
\label{cuarto}
\end{figure}

\begin{eje}\label{ej4}{\em 
Consider $p_1,p_2\in\r^2$, and $S_1,S_2$ two disjoint open sectors with vertex at $p_1,p_2$, respectively. A solution $u$ to (\ref{the-equation}) can be obtained in $\r^2\backslash\{p_1,p_2\}$ such that $u$ is a conical function in $S_1$ and $S_2$, and $u$ is linear in the open set determined by $\r^2\backslash(\overline{S_1}\cup\overline{S_2})$ (if not empty). See Figure \ref{quinto}.}
\end{eje}
\begin{figure}[!h]
\centerline{\includegraphics[scale =0.3]{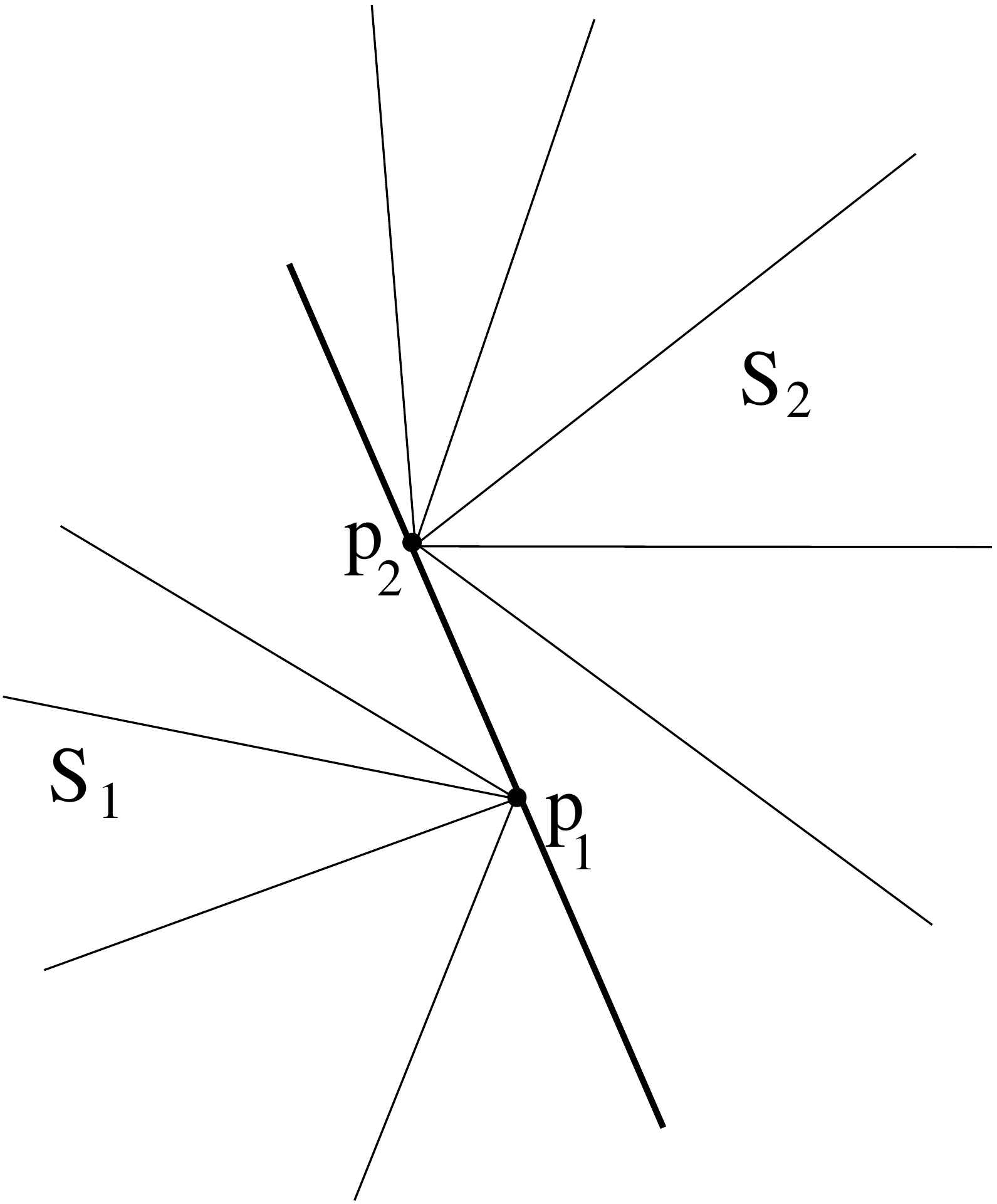}\hspace{1cm} \includegraphics[scale =0.3]{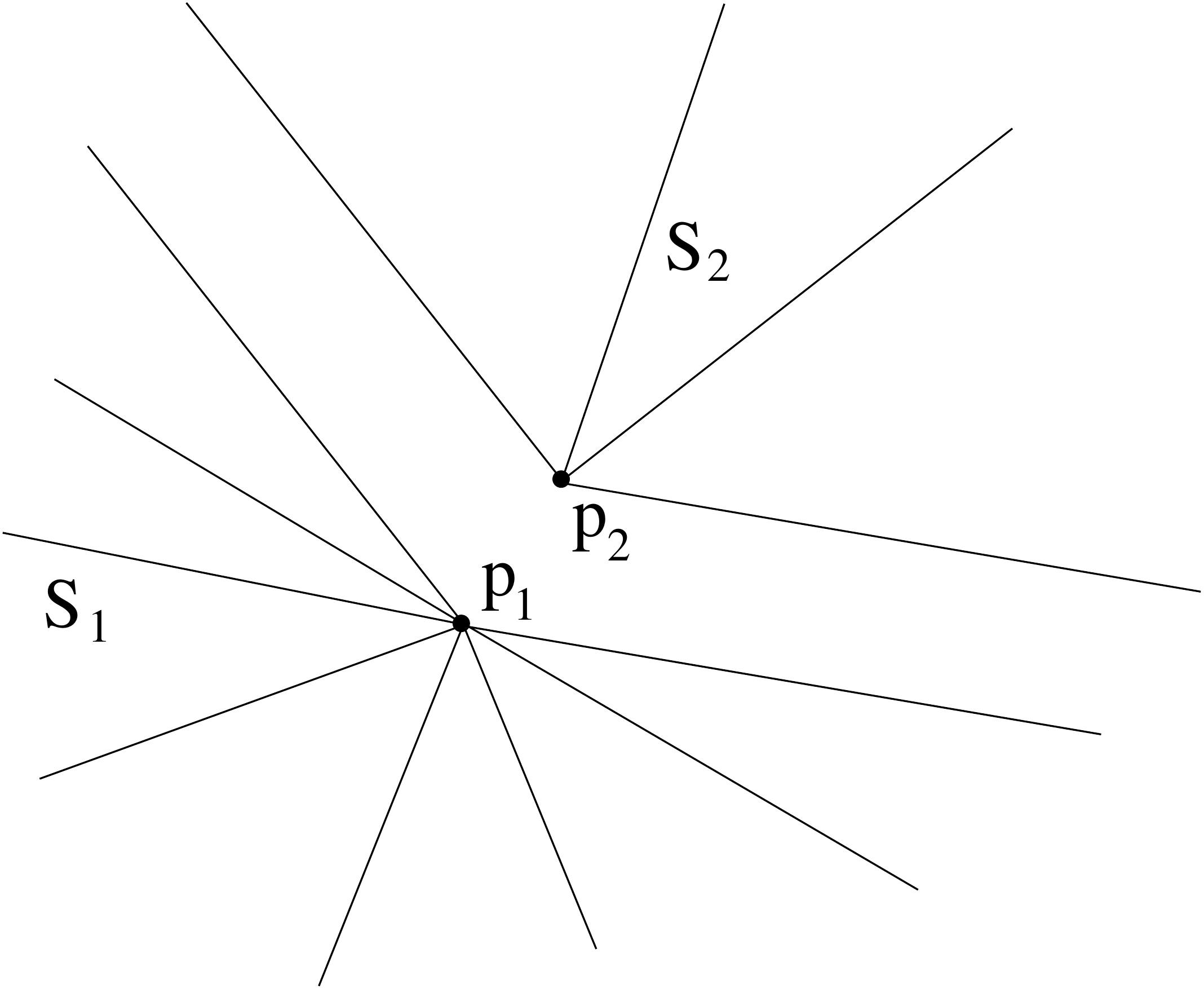}}
\caption[]{the configurations of $r(p)$ lines for example \ref{ej4}.}
\label{quinto}
\end{figure}

Our objective is to show that every solution to  (\ref{the-equation}) in the twice punctured plane is given by one of the previous examples. For that, we will need some properties associated with the sets $\n_{p_i}$ defined in (\ref{npi}).

Let $u:\r^2\backslash\{p_1,\ldots,p_n\}\fl\r$ be a ${\cal C}^2$ solution of (\ref{the-equation}) with isolated singularities at the points $p_i$. We define for each singular point
$$
{\cal V}_{p_i}=\{v\in\s^1:\text{ there exists }p\in\n_{p_i}\text{ such that }r(p)=p_i+\lambda v, \text{ with }\lambda>0\}.
$$
Observe that ${\cal V}_{p_i}$ is a non empty open set in the unit circle since $\n_{p_i}$ is a non empty open set (Corollary \ref{c1}).
\begin{lem}\label{desconexion}
Let $u:\r^2\backslash\{p_1,\ldots,p_n\}\fl\r$ be a ${\cal C}^2$ solution of (\ref{the-equation}) with isolated singularities at the points $p_i$. Consider two different singular points $p_j,p_k$, then there exist two disjoint connected open sets $U_j,U_k\subseteq\s^1$ such that ${\cal V}_{p_j}\subseteq U_j$ and ${\cal V}_{p_k}\subseteq U_k$.
\end{lem}
\begin{proof}
Let $u_0=\dfrac{p_k-p_j}{|p_k-p_j|}$ and $v_0,v\in{\cal V}_{p_j}$. It is clear that $v_0$ and $v$ are different from $u_0$ because the half-line $p_j+\lambda u_0,\lambda>0,$ contains the singular point $p_k$ (see Figure \ref{primera}). If $v_0\neq v$, the set $\s^1\backslash\{v_0,v\}$ has two connected components. Let $A_v$ be the connected component containing the unit vector $u_0$, then we want to show that ${\cal V}_{p_k}\subseteq A_v$.

Observe that $p_j$ and the half-lines $p_j+\lambda v_0$, $p_j+\lambda v$, with $\lambda>0$, separate the plane into two connected components. Let $\Omega$ be the connected component containing $p_k$. Then, for $w\in\s^1$, it is easy to check that if the half-line $p_k+\lambda w$, $\lambda>0$, is contained in $\Omega$ then $w\in A_v\cup\{v_0,v\}$ (Figure \ref{primera}). As a consequence, ${\cal V}_{p_k}$ must be contained in $A_v$.

Thus, ${\cal V}_{p_k}$ is contained in $A=\cap_{v\in{\cal V}_{p_j}}A_{v}$. Since the intersection of convex sets of a punctured circle (which is homeomorphic to $\r$), $\s^1\backslash\{v_0\}$, must be convex, and since ${\cal V}_{p_k}$ is an open set of $\s^1$, we have that ${\cal V}_{p_k}$ is contained in the open connected set given by the interior of $A$. Moreover, by construction, the open set ${\cal V}_{p_j}$ is contained in $\s^1\backslash\overline{A}$.
\end{proof}
\begin{figure}[!h]
\centerline{\includegraphics[scale =0.3]{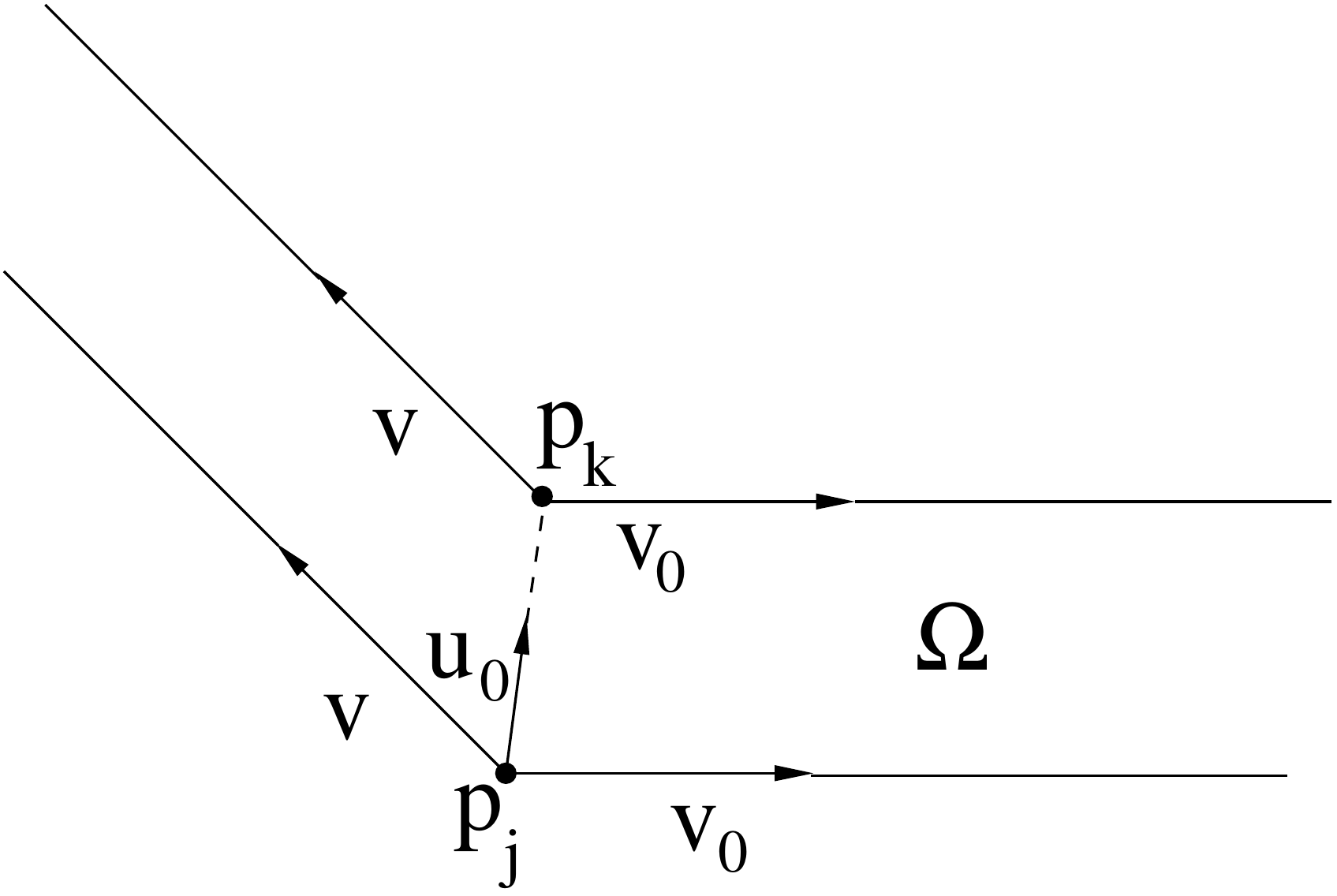}\hspace{1.5cm} \includegraphics[scale =0.3]{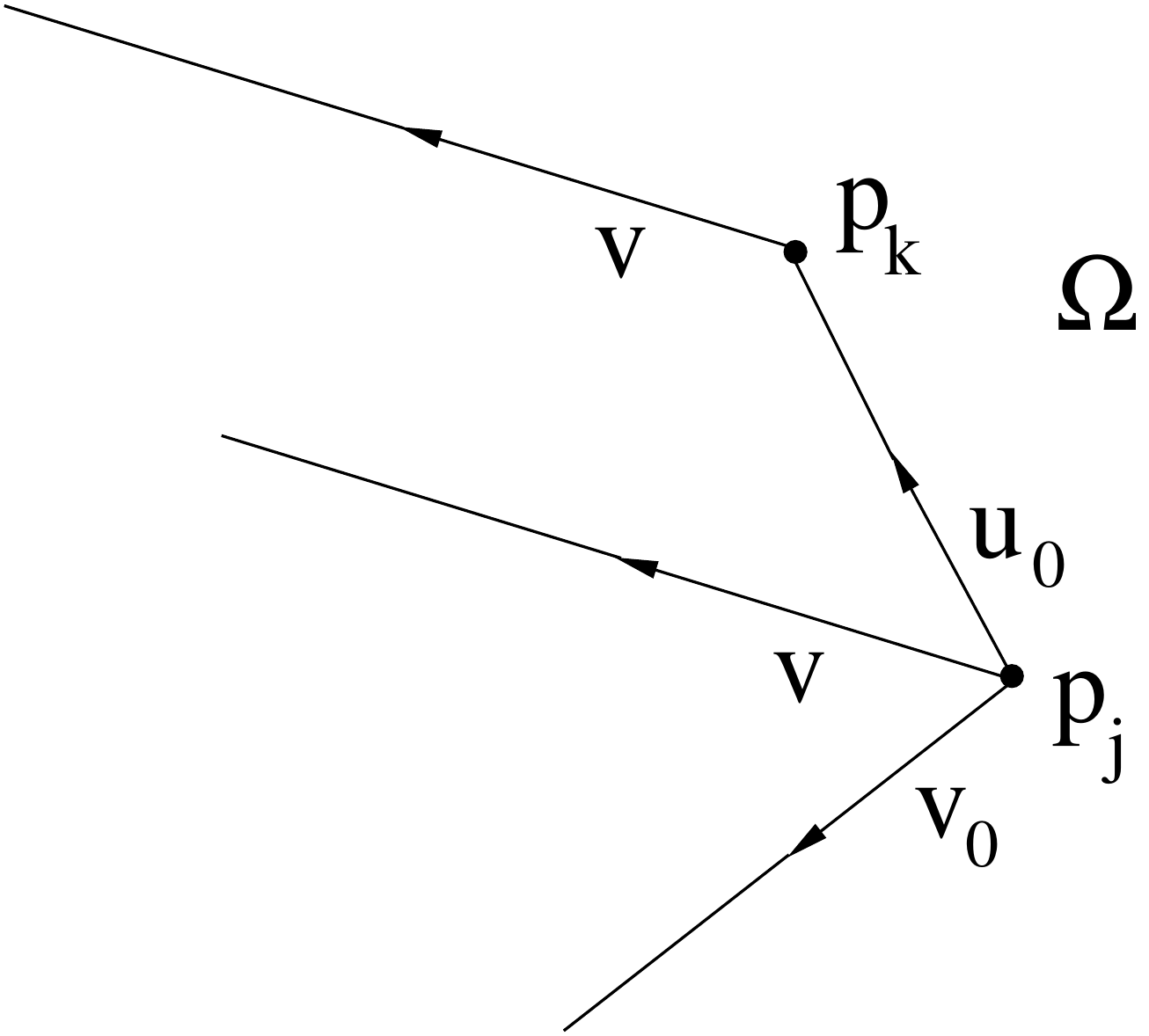} }
\caption[]{proof of Lemma \ref{desconexion}.}
\label{primera}
\end{figure}
\begin{teo}
Let $u:\r^2\backslash\{p_1,p_2\}\fl\r$ be a ${\cal C}^2$ solution of (\ref{the-equation}) with two singular points $p_1,p_2$. Then, $u$ is contained in one of the families given by Examples \ref{ej1}, \ref{ej2}, \ref{ej3} or \ref{ej4}.
\end{teo}
\begin{proof}
We will divide the proof in different cases:

{\it Case 1:} Assume there exists a point $p\in\n$ such that $r(p)$ is a line. Then, from Proposition \ref{max-strip}, there exists a maximal open strip $S(p)$ which is contained in $\r^2\backslash\{p_1,p_2\}$ such that each connected component of its boundary in $\r^2$ has at least a singular point. Moreover, $u$ is a cylindrical function in $S(p)$.

{\it Case 1.1:} If the boundary of $S(p)$ is made by two lines, then each line must contain one of the two singular points. In addition, if a point $q\in\n$ is not in the strip $S(p)$ then $r(q)$ is a half-line with end point at a singular one. Hence, from Lemma \ref{l3}, $u$ is a conical function in each half-plane of $\r^2\backslash\overline{S(p)}$ and is determined by Example \ref{ej1}.

{\it Case 1.2:} If the boundary of $S(p)$ in $\r^2$ is only one line, then two different situations can happen. Up to a translation or rotation, we can assume $S(p)$ is the open half-plane $y<0$:

{\it Case 1.2.1:} If both singularities are in the boundary of $S(p)$ then we can write $p_1=(x_1,0)$, $p_2=(x_2,0)$, with $x_1<x_2$. Then, from Lemma \ref{desconexion}, there exists $\theta_0\in(0,\pi)$ such that $\n_{p_2}$ is contained in the open sector with vertex at $p_2$ and angles between $0$ and $\theta_0$:
$$
S_0^{\theta_0}(p_2)=\{p_2+\rho\,(\cos\theta,\sin\theta):\rho>0,\ 0<\theta<\theta_0\},
$$
and $\n_{p_1}$ is contained in the open sector $S_{\theta_0}^{\pi}(p_1).$ So, $u$ is a conical function in $S_0^{\theta_0}(p_2)$ and also in $S_{\theta_0}^{\pi}(p_1).$ Hence, $u$ is given as in Example \ref{ej2}.

{\it Case 1.2.2:} If only one singularity is in the boundary of $S(p)$, then we have $p_1=(x_1,0)$,  and $p_2=(x_2,y_2)$, with $y_2>0$. Using again Lemma \ref{desconexion}, there exist $\theta_1,\theta_2\in[0,\pi]$, $\theta_1<\theta_2$, such that $\n_{p_2}\subseteq S_{\theta_1}^{\theta_2}(p_2)$ and $\n_{p_1}\subseteq S_{0}^{\theta_1}(p_1)\cup S_{\theta_2}^{\pi}(p_1)$. Here,  $S_{0}^{\theta_1}(p_1)$, $ S_{\theta_2}^{\pi}(p_1)$ could be the empty set if $\theta_1=0$ or $\theta_2=\pi$, respectively, but both of them cannot be empty since $\n_{p_1}\neq\emptyset$. Thus, $u$ belongs to the family determined by Example \ref{ej3}.

{\it Case 2:} If there is no $p\in\n$ so that $r(p)$ is a line then, from Lemma \ref{desconexion}, there exist angles $\theta_1<\theta_2\leq\theta_3<\theta_4\leq\theta_1+2\pi$, satisfying $\n_{p_1}\subseteq S_{\theta_1}^{\theta_2}(p_1)$ and $\n_{p_2}\subseteq S_{\theta_3}^{\theta_4}(p_2)$. Therefore, $u$ is a conical function in both sectors $S_{\theta_1}^{\theta_2}(p_1)$, $S_{\theta_3}^{\theta_4}(p_2)$, and is given as in Example \ref{ej4}.
\end{proof}

\section{Entire solutions with more than two  singular points.}\label{s5}

The results in Section \ref{s3} establish that every solution $u$ to the degenerate Monge-Ampère equation (\ref{the-equation}) in \break $\r^2\backslash\{p_1,\ldots,p_n\}$ must be a conical function in at least a sector with vertex at each singular point $p_i$ and, in addition, there exists at most one maximal strip where $u$ is a cylindrical function. However, when the number of singular points is  larger  than two, the number of different solutions to (\ref{the-equation}) is very large. This is basically due to the fact that one can find solutions $u$ to (\ref{the-equation}) defining $u$ as a linear function in some half-strips of the domain in such a way that $u$ is glued to the conical parts of the function, or its cylindrical part.

We will focus our attention on the classification of the solutions to (\ref{the-equation}) which do not contain a half-strip where the graph is planar.
\begin{defi}{\em 
Let $u:\r^2\backslash\{p_1,\ldots,p_n\}\fl\r$ be a ${\cal C}^2$ solution of (\ref{the-equation}) with isolated singularities at the points $p_i$. We say that $u$ is an {\em admissible solution} if there is no half-strip in $\r^2\backslash\{p_1,\ldots,p_n\}$ where $u(x,y)$ is linear.}
\end{defi}
\begin{lem}\label{angulo}
Let $u:\r^2\backslash\{p_1,\ldots,p_n\}\fl\r$ be an admissible ${\cal C}^2$ solution of (\ref{the-equation}) with isolated singularities at the points $p_i$. Assume there exists a closed sector $S$ in $\r^2$ of angle less than $\pi$ with vertex at a fixed singular point $p_{j_0}$, such that
 $\partial S\cap\n=\emptyset$. Moreover, suppose $\n_{p_j}\subseteq S$ for every singular point $p_j\in S\backslash\{p_{j_0}\}$.

Then, there is no singular point in $S$ different from $p_{j_0}$.
\end{lem}
\begin{proof}
Up to a translation and rotation, we can assume
$$
S=\{\rho\,(cos\theta,\sin\theta):\rho\geq0, \theta_0\leq\theta\leq\pi-\theta_0\}
$$
for a certain $\theta_0\in(0,\pi/2)$, with singular point $p_{j_0}=(0,0)$.

Observe that if $p\in S\cap\n$ then $r(p)$ can only be a half-line with end point at a singular point of $S$, since $\partial S\cap\n=\emptyset$. Moreover, $r(p)$ must be given as $p_j+\lambda (\cos\theta,\sin\theta)$ for some singular point $p_j\in S$, $\lambda>0$, and $\theta_0\leq\theta\leq\pi-\theta_0$.  In fact, $\theta\neq \theta_0,\pi-\theta_0$ because $\n_{p_j}$ is an open set.

Suppose there is a singular point in $S$ different from $p_{j_0}=(0,0)$. We choose $L>0$ big enough in such a way that every singular point is below the horizontal line $y=L$, and consider a highest singular point $p_k=(x_k,y_k)\in S$. For this point, the set $\{x\in\r: (x,L)\in\n_{p_k}\}$ is bounded from below and from above, since $\n_{p_k}\subseteq S$. Thus, we define the real numbers
$$
m=\inf\{x\in\r: (x,L)\in\n_{p_k}\},\qquad M=\sup\{x\in\r: (x,L)\in\n_{p_k}\}.
$$

Note that if $(m,L)\in\partial S$ then $p_k$ belongs to the segment joining $(m,L)$ and the vertex $p_{j_0}=(0,0)$, which is contained in $\partial S$. Otherwise, by definition of $m$ we would have $\n_{p_k}\cap\partial S\neq\emptyset$, which contradicts $\n\cap\partial S=\emptyset$. Analogously, if $(M,L)\in\partial S$ then $p_k$ belongs to the segment joining $(M,L)$ and the vertex $p_{j_0}$. In particular, $(m,L)$ and $(M,L)$ are not in $\partial S$ at the same time.

{\it Claim 1}: If $(m,L)\not\in\partial S$ then there exists a singular point $p_m\in S$ in the line $(1-\lambda)p_k+\lambda(m,L)$, with $\lambda<0$, such that $(m,L)\in\overline{\n_{p_m}}$ (see Figure \ref{pm}).

\begin{figure}[!h]
\centerline{\includegraphics[scale =0.3]{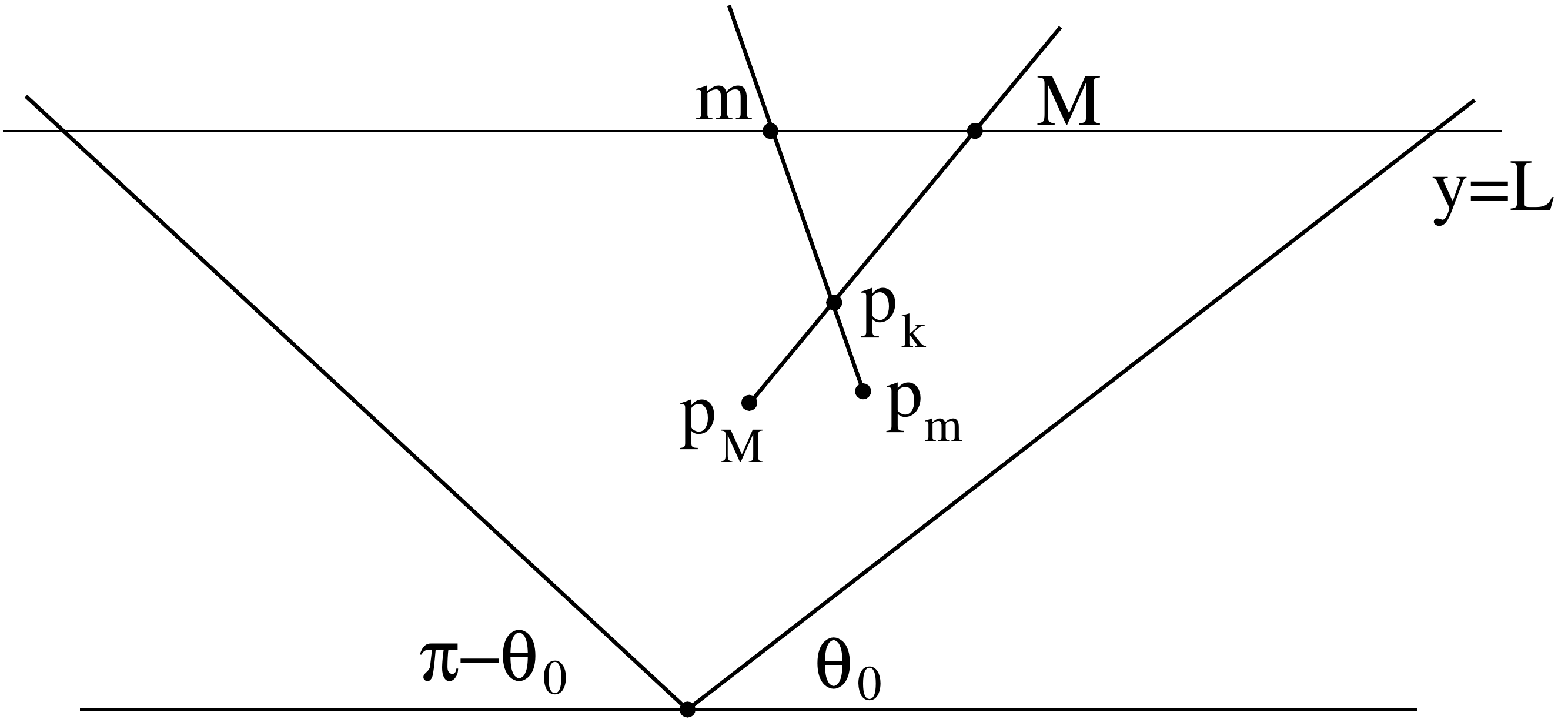}}
\caption[]{proof of Claim 1 of Lemma \ref{angulo}}
\label{pm}
\end{figure}

In order to prove the claim, consider the set ${\cal I}=\{x\in\r: x<m, (x,L)\in\n\cap S\}$. ${\cal I}$ cannot be empty, otherwise the domain in $S$ determined by the half-line $(1-\lambda)p_k+\lambda(m,L)$, with $\lambda\geq1$, the segment $\{(x,L)\in S: x\leq m\}$ and $\partial S$, would give rise to an unbounded domain included in $\u$, since its boundary lies in $\u$ and every singular point is outside of this domain. As the domain contains a half-strip, the solution would not be admissible.

Let $\overline{m}$ be the supremum of ${\cal I}$, then we want to see that $\overline{m}=m$. If $\overline{m}<m$ then the half-line $r_m$ given by $(1-\lambda)p_k+\lambda(m,L)$, with $\lambda\geq1$, its parallel half-line $r_{\overline{m}}$ with end point $(\overline{m},L)$, and the segment $\{(x,L)\in S: \overline{m}\leq x\leq m\}$ will
determine a half-strip $H$ in $\u$. For that, it is enough to show that $r_{\overline{m}}$ is contained in $\u$, since, as above, $\partial S$ would be included in $\u$ and every singular point would be outside $H$. Hence, $H\subseteq\u$ and the solution would not be admissible.

If there exists $p\in r_{\overline{m}}\cap\n$ then $r(p)$ is given by $p_j+\lambda v_0$, for a certain singular point $p_j\in S$, $\lambda>0$, $v_0\in\s^1$. Observe that $r(p)$ cannot intersect $r_m$ or the segment $\{(x,L)\in S: \overline{m}\leq x\leq m\}$. However, $p_j$ is below the line $y=L$, so if $r(p)\cap r_{\overline{m}}\neq\emptyset$ then $r(p)$ must intersect the parallel half-line $r_m$, which is a contradiction. Therefore, $\overline{m}=m$.

From the definition of $\overline{m}=m$, there exists a sequence of points $q_i=(x_i,L)\in\n$ tending to $(m,L)$ with $x_i<m$. Passing to a subsequence, if necessary, we can assume all $r(q_i)$ have a
common singular end point $p_m\in S$. This shows that $(m,L)\in\overline{\n_{p_m}}$.

Moreover, since $r(q_i)$ cannot intersect the half-line $(1-\lambda)p_k+\lambda(m,L)$, with $\lambda\geq0$, then $p_m$ belongs to the line $(1-\lambda)p_k+\lambda(m,L)$, with $\lambda\in\r$. But, $p_k$ is a highest point, so $p_m$ is below $p_k$, which proves the claim.

In an analogous way to the previous claim, one can prove:

{\it Claim 2}: If $(M,L)\not\in\partial S$ then there exists a singular point $p_M\in S$ in the line $(1-\lambda)p_k+\lambda(M,L)$, with $\lambda<0$, such that $(M,L)\in\overline{\n_{p_M}}$ (see Figure \ref{pm}).

Finally, in order to contradict the existence of the singular point $p_k$ and so conclude the proof, we distinguish two cases.

Assume $p_k\in\partial S$, then $p_k=\rho(\cos\theta,\sin\theta)$, with $\rho>0$ and $\theta\in\{\theta_0,\pi-\theta_0\}$. For instance, if $\theta=\theta_0$ then, as we proved previously, $(m,L)\not\in\partial S$ and, from Claim 1, there would exist a singular point $p_m\in S$ in the line $(1-\lambda)p_k+\lambda(m,L)$, with $\lambda<0$, such that $(m,L)\in\overline{\n_{p_m}}$. But this is impossible because every point in the half-line $(1-\lambda)p_k+\lambda(m,L)$ with $\lambda<0$ is contained in $\r^2\backslash S$, which contradicts the existence of $p_m$.

Assume $p_k\not\in\partial S$, then there exist two different singular points $p_m,p_M$ in the conditions of Claim 1 and Claim 2, respectively. But, this is a contradiction since $\n_{p_m}\cap\n_{p_M}\neq\emptyset$ (see Figure \ref{pm}).
\end{proof}

As a consequence we will prove that an admissible solution to (\ref{the-equation}), with a point $p\in\n$ such that $r(p)$ is a line, must be given by Case 2 in Theorem \ref{t2} or by Example \ref{ej1}.
\begin{cor}\label{c2}
Let $u:\r^2\backslash\{p_1,\ldots,p_n\}\fl\r$ be an admissible ${\cal C}^2$ solution of (\ref{the-equation}) with isolated singularities at the points $p_i$. If there exists a point $p\in\n$ such that $r(p)$ is a line, then there are exactly one or two singular points.
\end{cor}
\begin{proof}
From Proposition \ref{max-strip}, there exists an open strip $S(p)$ containing $r(p)$ such that its boundary in $\r^2$ is given by one or two lines parallel to $r(p)$. Moreover, if $r_0$ is a line of $\partial S(p)$ then $r_0$ contains at least a singular point $p_{i_0}$.

We need to show that if $H_0$ is the closed half-plane determined by $r_0$ which does not contain $S(p)$ then the unique singular point in $H_0$ is $p_{i_0}$.

Let $V$ be a connected component of $\n_{p_{i_0}}$. The set $V$ is an open sector contained in $H_0$ with vertex at $p_{i_0}$.

If $H_0\backslash V=r_0$ then there is no singular point different from $p_{i_0}$. Otherwise, if there exists another singular point $p_j\in H_0$ then $p_j\in r_0$ and the non empty set $\n_{p_j}$ would intersect $\n_{p_{i_0}}$ or $S(p)$.

If $H_0\backslash V\neq r_0$, then $H_0\backslash (V\cup\{p_{i_0}\})$ has one or two connected components with non empty interior. The adherence of each component with non empty interior is a closed sector with vertex at $p_{i_0}$ in the conditions of Lemma \ref{angulo}. Therefore, there is no singular point in $H_0$ except $p_{i_0}$.
\end{proof}

As a main result in this Section, we classify all the admissible solutions to the degenerate Monge-Ampère equation with more than two singularities (see Figure \ref{four-vertex}).
\begin{figure}[!h]
\centerline{\includegraphics[scale =0.5]{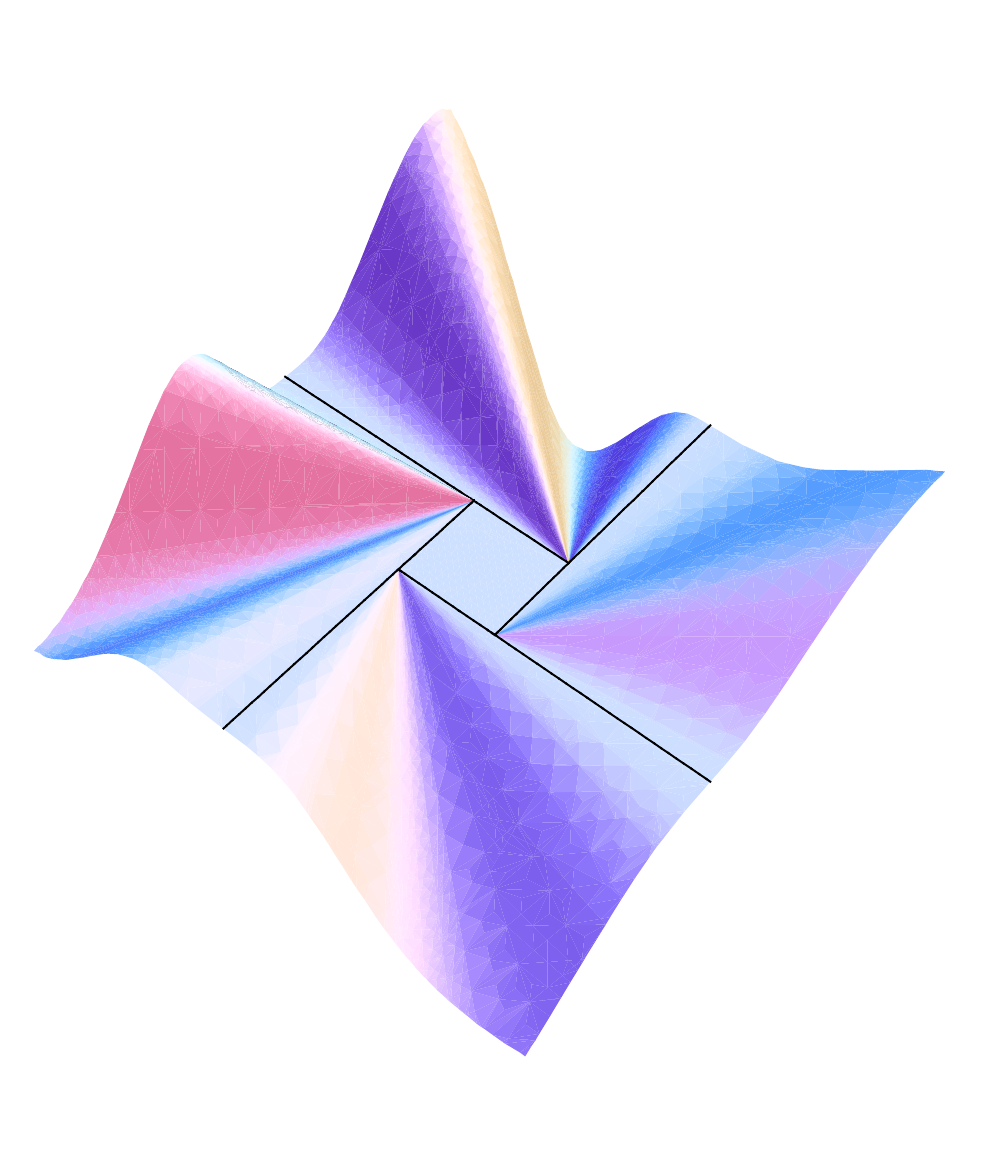}}
\caption[]{Admissible solution with four singular points.}
\label{four-vertex}
\end{figure}
\begin{teo}
Let $u:\r^2\backslash\{p_1,\ldots,p_n\}\fl\r$ be an admissible ${\cal C}^2$ solution of (\ref{the-equation}) with isolated singularities at the points $p_i$ and $n\geq3$. Then,
\begin{enumerate}
\item The singular points are the vertices of a convex compact polyhedron $C$ with non empty interior, and $u$ is linear on $C$.
\item The singular points can be numbered following an orientation of $\partial C$ in such a way that $u$ is a conical function on the open sector determined by the half-lines
$$
(1-\lambda)p_i+\lambda p_{i+1},\ \lambda\geq 0,\qquad\text{and}\qquad (1-\lambda)p_{i-1}+\lambda p_{i},\ \lambda\geq 1,
$$
with vertex at $p_i$. (Here, $p_0:=p_n$ and $p_{n+1}:=p_1$.)
\end{enumerate}
\end{teo}
\begin{proof}
Observe that, from Proposition \ref{l5} and Corollary \ref{c2}, if $p\in\n$ then $r(p)$ must be a half-line, that is, $\n=\cup_{i=1}^n {\n}_{p_i}$.

Let ${\cal S}$ be a circle of $\r^2$ such that the singular points are contained  in the interior of the bounded component determined by ${\cal S}.$ Let $${\cal S}_{p_i}=\overline{\n_{p_i}\cap {\cal S}}.$$
We want to show that ${\cal S}=\cup_{i=1}^n {\cal S}_{p_i}$.

Assume the open set ${\cal S}\backslash\cup_{i=1}^n {\cal S}_{p_i}$ is not empty, then we consider an open connected component  $U$ of this set. Let $q_1,q_2\in{\cal S}$ be the points in $\overline{U}\backslash U$. From the definition of ${\cal S}_{p_i}$ there exist two sequences $x_m^1, x_m^2\in{\cal S}\cap\n$ such that $x_m^j$ tends to $q_j$.

Passing to a subsequence if necessary, we can assume the half-lines $r(x_m^1)$ have the same singular end point $p_{i_1}$. Analogously, the half-lines $r(x_m^2)$ have a common singular end point $p_{i_2}$. But $U$ and the two half-lines
$$
(1-\lambda)p_{i_j}+\lambda q_j, \ \lambda\geq1,\quad j=1,2,
$$
determine a domain $V\subseteq\r^2$ whose boundary is included in $\u$ and every singular point is outside $V$. So, $V\subseteq \u$ and, since $V$ contains a half-strip, the solution $u$ cannot be admissible. This is a contradiction, hence ${\cal S}=\cup_{i=1}^n {\cal S}_{p_i}.$

Now, let us see that ${\cal S}_{p_i}$ is a connected set for each $i$.

Assume there exists $i_0$ such that ${\cal S}_{p_{i_0}}$ is not connected, then we can consider two different connected components $V_1,V_2$ of $\overline{\n_{p_{i_0}}}\backslash\{p_{i_0}\}$, where for instance $V_1$ has non empty interior (see Figure \ref{otramas}).

\begin{figure}[!h]
\centerline{\includegraphics[scale =0.3]{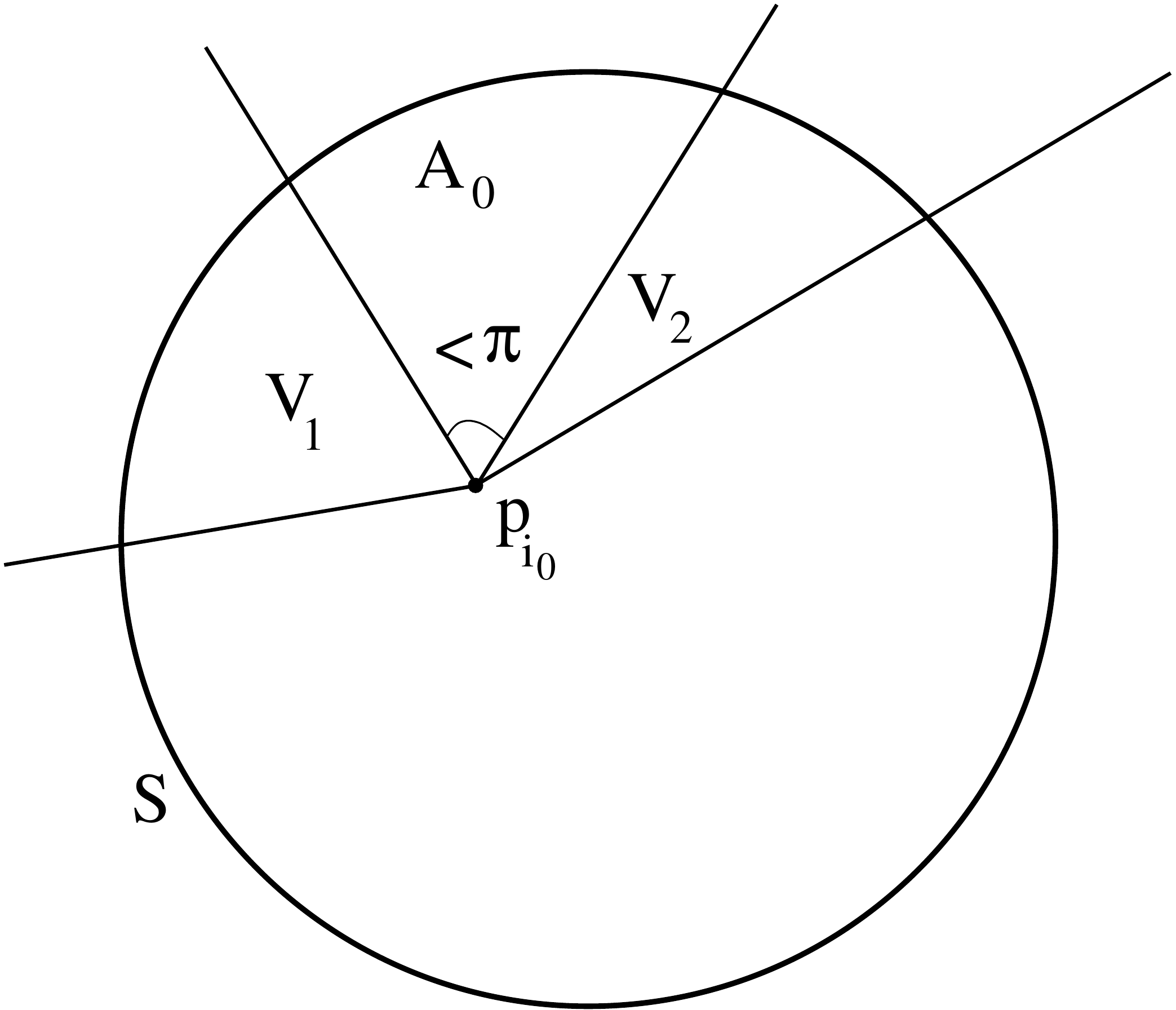}}
\caption[]{connection of  ${\cal S}_{p_{i_0}}$}
\label{otramas}
\end{figure}

Then $\r^2\backslash(V_1\cup V_2\cup\{p_{i_0}\})$ is made of two disjoint open sectors with vertex at $p_{i_0}$, where one of them, say $A_0$, has angle less than $\pi$, since the interior of $V_1$ is non empty. Using Lemma \ref{angulo} for the adherence of $A_0$ we obtain that the unique isolated singularity in $\overline{A_0}$ is $p_{i_0}$. Hence, $A_0\cap{\cal S}\subseteq {\cal S}_{p_{i_0}}$, because  ${\cal S}=\cup_{i=1}^n {\cal S}_{p_i}$. This contradicts that $V_1$ and $V_2$ are disjoint connected components.

As ${\cal S}_{p_i}$ is a connected set, we obtain that $\overline{\n_{p_i}}$ is a closed sector with vertex at $p_i$. Moreover, let us check that $\overline{\n_{p_i}}$ must have angle less than or equal to $\pi$. In order to see that, suppose the angle of $\overline{\n_{p_i}}$ is bigger than $\pi$, then using Lemma \ref{angulo} to the adherence of $\r^2\backslash\overline{\n_{p_i}}$ we would have that $p_i$ is the unique singular point, which contradicts the existence of at least three singularities.

Now,
following an orientation of $\mathcal S,$ we can adjust the indices in such a way that  ${\mathcal S}_{p_i}$ is adjacent to ${\mathcal S}_{p_{i-1}}$ and ${\mathcal S}_{p_{i+1}},$ $i=1,\dots, n.$
Let $\alpha_i,\beta_i\in \partial {\mathcal S}_{p_i}$ such that  $\alpha_i=\beta_{i-1},$ $i=1,\dots,n$ (see Figure \ref{convexity}).

\

{\it Claim 1.} Up to reversing the orientation of $\mathcal S$,  we can assume that, for any  $i=1,\dots, n$, the point $p_{i+1}$ belongs to the segment

\begin{equation*}
(1-\lambda)p_i+\lambda \beta_i, \  \ 0<\lambda<1.
\end{equation*}

We first prove that the point $p_2$ belongs to the half-line $(1-\lambda)p_1+\lambda \beta_1,$  $\lambda<1.$

There exists a sequence of point $x^2_m\in {\mathcal S}_{p_2}\cap{\mathcal N}$
such that $x^2_m$ tends to $\beta_1$  and $r(x^2_m)$ is a  half-line with endpoint  at $p_2.$ If $p_2$ does not belong to the  half-line  through $\beta_1$ and $p_1,$ with endpoint at $\beta_1,$
then  for $m$ sufficiently large, $r(x^2_m)$ would intersect the sector $\overline{{\mathcal N}_{p_1}}.$  This is a contradiction.

Now, if $p_2$   does not belong to the segment between $\beta_1$ and $p_1,$ then $p_1$  belongs to the segment between $p_2$ and
$\alpha_2=\beta_1$  and  we get  the desired result by changing the orientation of $\mathcal S.$

By the same argument used for the point $p_2,$ we can prove that $p_3$ belongs to the line trough   $\beta_2$ and $p_2.$
Moreover, $p_3$ must belong to the open segment between $\beta_2$ and $p_2,$ because on the contrary one has
$\overline{{\mathcal N}_{p_1}}\cap \overline{{\mathcal N}_{p_3}}\not=\emptyset,$  that is a contradiction.

Analogously, one can prove that for any $i=3,\dots,n-1,$ the point $p_{i+1}$  belongs to the open segment between $\beta_i$ and $p_i ,$ that proves Claim 1.

\

{\it Claim 2.} The points $p_1,\dots, p_n$  are the vertices of a convex compact polyhedron  $C.$

\begin{figure}[!h]
\centerline{\includegraphics[scale =0.3]{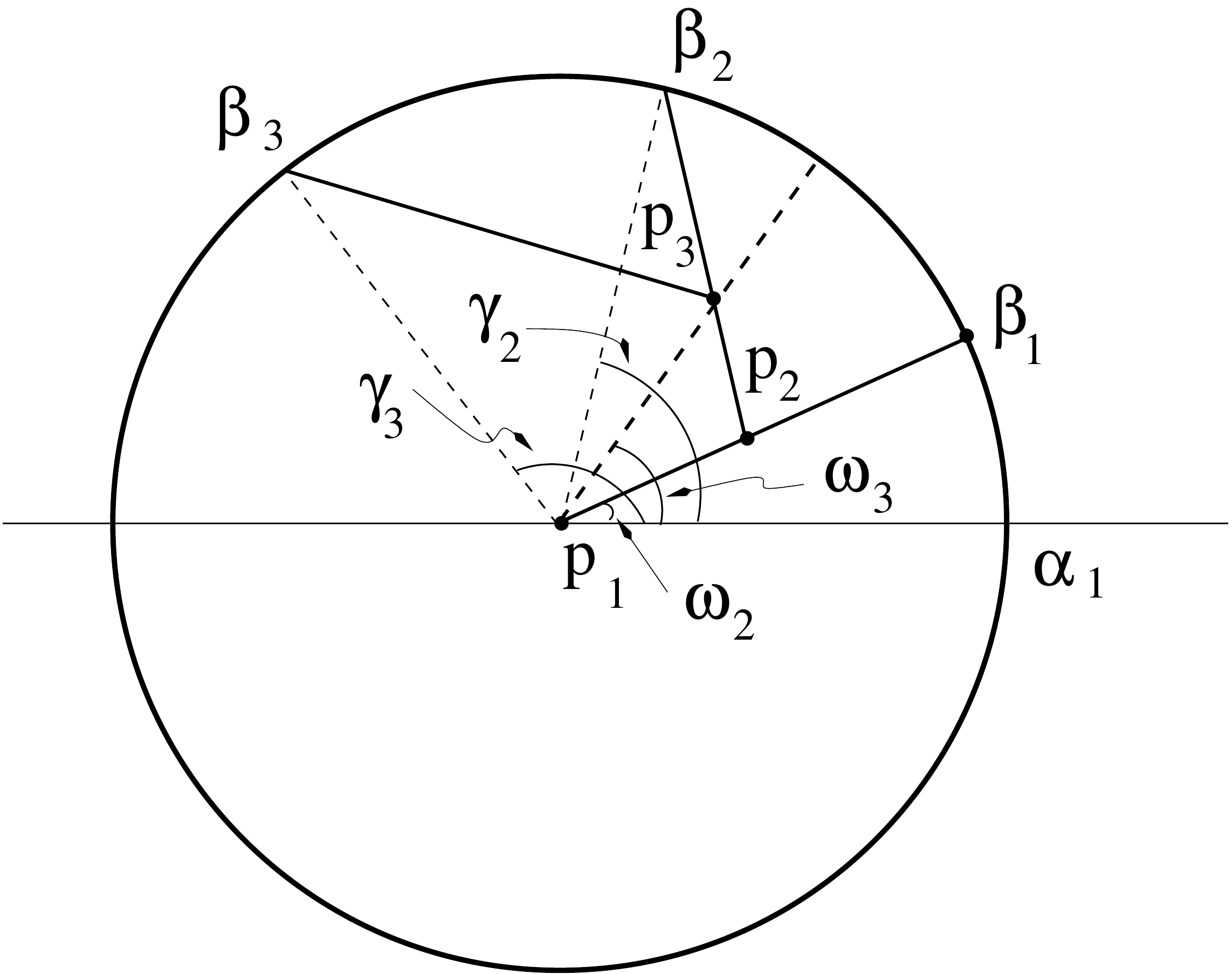}}
\caption[]{$C$ is convex}
\label{convexity}
\end{figure}

Up to a translation  and a rotation, we can assume that $p_1$ is the origin and that $\alpha_1$ is the intersection between the positive $x$-axis and $\mathcal S$ (see Figure \ref{convexity}).   For  $i=2,\dots,n,$ denote by $\omega_i$ the angle that the line through $p_1$ and $p_i$ does with the
$x$-axis   and by  $\gamma_i$ the angle that the line through $p_1$ and $\beta_i$ does with the
$x$-axis (both angles with the orientation induced by the orientation of $\mathcal S$). It is straightforward to  verify that

\begin{equation}
\label{angle-ineq}
0<\omega_2<\omega_3<\gamma_2<\omega_4<\gamma_3<\omega_5<\gamma_4<\omega_6<\dots<\omega_n<\gamma_n=2\pi
\end{equation}
where the last equality is  because $\beta_n=\alpha_1.$
In particular $\omega_n=\pi$, since $p_1$ is contained in the segment between $p_n$ and $\beta_n=\alpha_1$.

Notice that, inequality  \eqref{angle-ineq} implies that, for $i=2\dots,n-1,$ one has $\omega_i<\pi.$ Hence  all the $p_i,$  $i=2,\dots, n-1$ lie in the same open half-plane with boundary the  straight line through $p_1$ and $p_n.$ Analogously, for each $i=2,\ldots,n$ the straight line through $p_{i-1}$ and $p_i$ determines an open half-plane where the rest of singular points are included. This shows, that the intersection of the corresponding closed half-planes determine a convex compact polyhedron $C$ with non empty interior.

Finally, we have proved that, for any singular point $p_i,$ the subset ${\mathcal N}_{p_i}$  does not intersect the convex polyhedron $C.$ Since
${\mathcal N}=\cup_{i=1}^n {\mathcal N}_{p_i},$ one has that $C\subseteq {\mathcal U}$ and hence $u$ is linear in $C,$ as desired.
\end{proof}
Let us conclude with some remarks. On the one hand, an analogous result for global solutions to the degenerate Monge-Ampère equation with isolated singularities in higher dimensions would be desirable, although it seems to be more difficult. On the other hand, we believe that our method can be extended to determine the solutions to the degenerate Monge-Ampère equation in the 
 plane $\r^2$ with some convex sets  removed. In the elliptic case, some results in this sense are in \cite{JX}, where  the authors study Monge-Ampère equations with isolated and lines singularities. Finally, it should be mentioned that some interesting global results were obtained in \cite{MU} for flat surfaces where the set of singularities can be large.

\end{document}